\documentclass[10pt,reqno]{amsart}
\usepackage{amsmath, amsthm, amscd, amssymb, amsfonts, amsbsy}
\usepackage{latexsym}
\usepackage{color}
\usepackage{enumerate}
\usepackage{pxfonts}
\usepackage{verbatim}

\usepackage{geometry}
\theoremstyle{plain}
\newtheorem{theorem}[equation]{Theorem}
\newtheorem{lemma}[equation]{Lemma}
\newtheorem{corollary}[equation]{Corollary}
\newtheorem{proposition}[equation]{Proposition}

\theoremstyle{definition}
\newtheorem{definition}[equation]{Definition}

\theoremstyle{remark}
\newtheorem{remark}[equation]{Remark}

\newtheorem*{acknowledgments*}{Acknowledgments}

\DeclareMathOperator{\Lip}{Lip}

\providecommand{\set}[1]{\{#1\}}
\providecommand{\Set}[1]{\left\{#1\right\}}

\providecommand{\abs}[1]{\lvert#1\rvert}
\providecommand{\Abs}[1]{\left\lvert#1\right\rvert}
\providecommand{\bigabs}[1]{\bigl\lvert#1\bigr\rvert}

\providecommand{\ip}[1]{\langle#1\rangle}

\providecommand{\norm}[1]{\lVert#1\rVert}
\providecommand{\Norm}[1]{\left\lVert#1\right\rVert}
\providecommand{\bignorm}[1]{\bigl\lVert#1\bigr\rVert}
\providecommand{\Bignorm}[1]{\Bigl\lVert#1\Bigr\rVert}

\newcommand{\bR}{\mathbb{R}}

\DeclareMathOperator{\dv}{div}

\newcommand*{\tran}{^{\mkern-1.5mu\mathsf{T}}}

\renewcommand{\vec}[1]{\boldsymbol{#1}}

\numberwithin{equation}{section}

\begin{document}
\title[Green's function]{Green's function for second order elliptic equations with singular lower order coefficients}
\author[S. Kim]{Seick Kim}
\address[S. Kim]{Department of Mathematics, Yonsei University, 50 Yonsei-ro, Seodaemun-gu, Seoul 03722, Republic of Korea}
\email{kimseick@yonsei.ac.kr}
\thanks{S. Kim is partially supported by NRF Grant No. NRF-2016R1D1A1B03931680 and No. NRF-20151009350.}

\author[G. Sakellaris]{Georgios Sakellaris}
\address[G. Sakellaris]{Department of Mathematics, Universitat Aut{\`o}noma de Barcelona, Bellaterra 08193, Barcelona, Spain}
\email{gsakellaris@mat.uab.cat}
\thanks{G. Sakellaris has received funding from the European Union's Horizon 2020 research and innovation programme under Marie Sk{\l}odowska-Curie grant agreement No 665919, and is partially supported by MTM-2016-77635-P (MICINN, Spain) and 2017 SGR 395 (Generalitat de Catalunya).}



\begin{abstract}
We construct  Green's function for second order elliptic operators of the form
$Lu=-\nabla \cdot (\mathbf{A} \nabla u + \boldsymbol{b} u)+ \boldsymbol c \cdot \nabla u+ du$ in a  domain and obtain pointwise bounds, as well as Lorentz space bounds.
We assume that the matrix of principal coefficients $\mathbf A$ is uniformly elliptic and bounded and the lower order coefficients $\boldsymbol{b}$, $\boldsymbol{c}$, and $d$ belong to certain Lebesgue classes and satisfy the condition $d - \nabla \cdot  \boldsymbol{b} \ge 0$.
In particular, we allow the lower order coefficients to be singular.
We also obtain the global pointwise bounds for the gradient of Green's function  in the case when the mean oscillations of the coefficients $\mathbf A$ and $\boldsymbol{b}$ satisfy the Dini conditions and the domain is $C^{1, \rm{Dini}}$.
\end{abstract}

\maketitle

\section{Introduction}
Let $\Omega$ be a domain (i.e., an open connected set) in $\bR^n$ with $n \ge 3$.
We consider second order elliptic operators in divergence form
\[
L u= -\sum_{i,j=1}^n D_i(a^{ij}(x)D_ju+b^i(x)u)+ \sum_{i=1}^n c^i(x) D_i u + d(x)u,
\]
which hereafter shall be abbreviated as
\[
Lu=-\dv (\mathbf{A} \nabla u + \vec b u)+ \vec c \cdot \nabla u+ du.
\]
We assume that the principal coefficients $\mathbf{A}=(a^{ij})$ are measurable $n\times n$ matrices that are bounded and uniformly elliptic; i.e. there is a constant $\lambda >0$ such that
\begin{equation}			\label{ellipticity}
\lambda \abs{\vec \xi}^2 \le \mathbf{A}(x) \vec \xi \cdot \vec \xi = \sum_{i,j=1}^n a^{ij}(x) \xi^i \xi^j,\quad \forall x \in \Omega, \;\; \forall \vec  \xi \in \bR^n.
\end{equation}
We also assume that the lower order coefficients $\vec b=(b^1, \ldots, b^n)$, $\vec c=(c^1,\ldots, c^n)$, and $d$ are such that
\begin{equation}				\label{eq1547sat}
\vec b \in L^q(\Omega), \quad \vec c \in L^r(\Omega), \quad d \in L^s(\Omega)\quad\text{for some }\; q, r \ge n,\; s \ge n/2.
\end{equation}
It should be noted that the lower order coefficients are not assumed to be ''small" in some norm, and they are allowed to possess singularities.
Finally, we assume that
\begin{equation}				\label{eq1602sat}
d -\dv \vec b \ge 0
\end{equation}
in the sense of distributions.
We introduce the condition \eqref{eq1602sat} to make the weak maximum principle hold.
We remark that there is an example showing that the uniqueness for Dirichlet problem is violated if  the condition \eqref{eq1547sat} does not hold (even if \eqref{eq1602sat} is satisfied); see Ladyzhenskaya and Ural'tseva \cite[p. 13]{LU68}.

In this article, we are concerned with Green's function for the operator $L$ introduced above; the precise definition of Green's function is given in Definition~\ref{GreenDefinition}.
We shall show that if  $\abs{\Omega} < +\infty$, that is $\Omega$ has finite measure,  and $\vec b -\vec c \in L^p$ for $p>n$, then there exists a Green's function $G(x,y)$ that enjoys the pointwise bound
\begin{equation}					\label{eq2109sat}
\abs{G (x,y)} \le  C\abs{x-y}^{2-n},
\end{equation}
where $C$ is a constant that depends only on $n$, $p$, $\lambda$, $\norm{\mathbf A}_{\infty}$, $\norm{\vec b-\vec c}_p$, and $\abs{\Omega}\,$; see Theorem~\ref{GreenConstruction}.
We emphasize that, in this case, we require that $\vec b-\vec c$, not both $\vec b$ and $\vec c$, to be in $L^p$ for some $p>n$.
In fact, we also construct Green's function $g(x,y)$ for the adjoint operator and show the symmetry relation $G(x,y)=g(y,x)$; see Theorem \ref{GreenConstructionAdjoint} and Proposition \ref{Symmetry}. Moreover, for the gradient of $G$, in Theorem~\ref{GreenConstruction} we show a Lorentz type bound of the form
\begin{equation}					\label{2ndInIntroduction}
\norm{\nabla G(\cdot,y)}_{L^{\frac{n}{n-1},\infty}}\leq C,
\end{equation}
where $C$ depends only on $n$, $p$, $\lambda$, $\norm{\mathbf A}_{\infty}$, $\norm{\vec b-\vec c}_p$, and $\abs{\Omega}\,$.

In the critical case when $\vec b$, $\vec c \in L^n$, in Theorem~\ref{GoodLorentzForGreenAdoint} we obtain Lorentz type bounds on Green's function $g$ for the adjoint equation, as well as its gradient $\nabla g$; that is,
\begin{equation}					\label{3ndInIntroduction}
\norm{g(\cdot,x)}_{L^{\frac{n}{n-2},\infty}}\leq C,\quad \norm{\nabla g(\cdot,x)}_{L^{\frac{n}{n-1},\infty}}\leq C.
\end{equation}
Moreover, in order to deduce the bounds \eqref{eq2109sat} and \eqref{2ndInIntroduction} in the critical case, we shall also require, in addition to \eqref{eq1602sat}, that
\begin{equation}					\label{eq2118sat}
d-\dv \vec c \ge 0
\end{equation}
in the sense of distributions, which will make the weak maximum principle available for solutions to the adjoint equation as well; see Theorem~\ref{GreenConstructionAdjointCritical}.
It should be noted that the condition \eqref{eq2118sat} is not just a technical one because otherwise the bound \eqref{eq2109sat} may not hold; see Proposition~\ref{Counterexample}. The fact that \eqref{3ndInIntroduction} holds without the assumption \eqref{eq2118sat}, while \eqref{3ndInIntroduction} does not necessarily hold for $G$ without \eqref{eq2118sat}, exhibits the different nature of $G,g$ in the critical case.

In this article, we also obtain the gradient bound for Green's function when the mean oscillations of the coefficients $\mathbf{A}$ and $\vec b$ satisfy the  Dini condition and $\vec c$, $d \in L^p$ with $p>n$.
We show that if $\Omega$ is a bounded $C^{1,\rm{Dini}}$ domain, then  Green's function $G(x,y)$ satisfies the gradient bound
\begin{equation}				\label{eq1055am}
\abs{\nabla _x G(x,y)} \le  C\abs{x-y}^{1-n};
\end{equation}
see Theorem~\ref{GradientBounds}. In particular, this reproduces the gradient bound in Gr\"uter and Widman \cite{Gruter}, where it is assumed that $\vec b=0$, $\vec c=0$, $d=0$, and $\mathbf{A}$ is Dini continuous; if $\mathbf{A}$ is Dini continuous, then its mean oscillation satisfies the Dini condition, but not the other way around.
We believe that our condition is the minimal one available in the literature to have the pointwise gradient bound \eqref{eq1055am}, and it is one of the novelties in our paper.

A few historical remarks are in order.
Green's functions for second order elliptic operators $L$ of the form $Lu=-\dv (\mathbf{A} \nabla u)$ with bounded measurable coefficients $\mathbf{A}$ were studied by Littman, Stampacchia, and  Weinberger \cite{Littman} and later also by Gr\"uter and Widman \cite{Gruter}.
In Littman et al. \cite{Littman}, the matrix $\mathbf{A}$ is assumed to be symmetric so that $L$ becomes self-adjoint, while $\mathbf{A}$ is allowed to be non-symmetric in \cite{Gruter}.
In both \cite{Littman} and \cite{Gruter}, the domain $\Omega$ is assumed to be bounded. More recently, Hofmann and Lewis \cite[Chapter III, Lemma 4.3]{HofmannLewis} also constructed Green's function in a different setting, where they assume the so-called Bourgain condition on the harmonic measure (estimate (4.2), Chapter III in \cite{HofmannLewis}). Moreover, Hofmann and the first named author \cite{HK07} considered the case when the domain $\Omega$ is not necessarily bounded (including $\Omega = \bR^n$) and constructed Green's functions.
Although the method used in \cite{HK07} is applicable to strongly elliptic systems, the key for obtaining the pointwise bound \eqref{eq2109sat} is De Giorgi-Moser-Nash type estimates, which is not always available for the elliptic systems; it should be also noted that lower order coefficients are not present in \cite{HK07}.  

In a very recent article \cite{MayborodaGreen}, Davey, Hill, and Mayboroda considered elliptic systems with lower order coefficients and, under a series of assumptions, constructed Green's matrices with the same pointwise bounds as in \eqref{eq2109sat}. Even though we consider only the scalar case, there are some differences between our results and the scalar counterparts considered in \cite{MayborodaGreen}; the novelty in our paper is as follows:
\begin{enumerate}[i)]
\item
We treat the critical case when $\vec b \in L^n$, $\vec c \in L^n$, and $d \in L^{n/2}$.
\item
We do not require $d- \dv \vec c \ge 0$ if $\vec b - \vec c \in L^p$ for $p>n$.
\end{enumerate}

Finally, we remark that the results presented here are generalizations of some results in the second author's Ph.D. thesis \cite{SakellarisThesis}, in which bounded drifts $\vec b$ were considered, and $\vec c=\vec 0,d=0$.

The organization of the paper is as follows.
In Section~\ref{sec2}, we introduce some notations and preliminary lemmas.
In Section~\ref{sec3}, we state and prove lemmas concerning a priori estimates including a (local) maximum principle for (sub)solutions.
In Section~\ref{sec4}, we first establish existence and uniqueness of weak solutions to $Lu=F$ and $L\tran u=F$ for $F\in W^{-1,2}$, and then show boundedness properties of the solution operators $T=L^{-1}$ and $S=(L\tran)^{-1}$.
In Section \ref{sec5}, we introduce the precise definition of Green's function and provide its construction via a sequence of approximate Green's functions, where Dirac delta functions are replaced by approximate identities.
We also establish uniform $L^{\frac{n}{n-2},\infty}$ estimates for approximate Green's functions, which are used in subsequent Sections~\ref{sec6} and \ref{sec7} to obtain the pointwise and Lorentz bounds for Green's functions; in Section~\ref{sec6}, we treat the subcritical case (when $\vec b -\vec c \in L^p$ for $p>n$) and in Section~\ref{sec7}, we consider the critical case.
Finally, Section~\ref{sec8} is devoted to a study of pointwise gradient bounds for Green's function when the coefficients $\mathbf A$ and $\vec b$ are of Dini mean oscillation. 

\begin{acknowledgments*}
We would like to thank the anonymous referee for his/her helpful comments and suggestions that improved the exposition of this paper.
We would also like to thank Steve Hofmann for his interest in this paper.
\end{acknowledgments*}

\section{Preliminaries}				\label{sec2}
\subsection{Definitions}
For $p\in[1,\infty)$, $W^{1,p}(\Omega)$ will denote the usual Sobolev space of functions $f\in L^p(\Omega)$, such that their distributional derivatives belong to $L^p(\Omega)$, with norm
\[
\norm{u}_{W^{1,p}(\Omega)}=\norm{u}_{L^p(\Omega)}+ \norm{\nabla u}_{L^p(\Omega)}.
\]
We shall often write $\norm{u}_p$ for $\norm{u}_{L^p(\Omega)}$ for the sake of simplicity.
We shall denote by $\Lip(\Omega)$ the set of all Lipschitz continuous functions in $\Omega$ and by $C_c^{\infty}(\Omega)$ the set of all indefinitely differentiable functions that are compactly supported in $\Omega$.
The space $W_0^{1,p}(\Omega)$ is the closure of $C_c^{\infty}(\Omega)$ in $W^{1,p}(\Omega)$.
Finally, $W^{-1,p'}(\Omega)$ is the dual space to $W_0^{1,p}(\Omega)$, where $p'$ is the conjugate exponent to $p$; i.e. $\frac{1}{p}+\frac{1}{p'}=1$.

For $F\in W^{-1,2}(\Omega)$, we say that $u\in W_0^{1,2}(\Omega)$ is a (weak) solution to the equation $Lu=F$ in $\Omega$, if for all $\phi \in C^\infty_c(\Omega)$, we have
\[
\int_{\Omega}\mathbf{A}\nabla u\cdot\nabla\phi+\vec b\cdot \nabla\phi \,  u+\vec c \cdot \nabla u\, \phi+du\phi=\ip{F,\phi}.
\]
From the Sobolev embedding $W_0^{1,2}(\Omega)\hookrightarrow L^{\frac{2n}{n-2}}(\Omega)$, this is equivalent to the same equality holding for all $\phi\in W_0^{1,2}(\Omega)$.
Hereafter, we shall denote by $C_n$ the constant appearing in the Sobolev inequality
\[
\norm{u}_{\frac{2n}{n-2}} \le C_n \norm{\nabla u}_{2},\quad \forall u \in W^{1,2}_0(\Omega).
\]
We say that $u\in W_0^{1,2}(\Omega)$ is a subsolution to the equation $Lu=F$ in $\Omega$ if
\[
\int_{\Omega}\mathbf{A}\nabla u \cdot \nabla\phi+\vec b \cdot \nabla\phi \, u+\vec c \cdot \nabla u \, \phi+du\phi\le \ip{F,\phi},
\]
for all $\phi\in C_c^{\infty}(\Omega)$ with $\phi\ge 0$.
The distributional definition for Green's function $G(x,y)$ would be that $G(\cdot, y)$ satisfies the identity
\[
\int_{\Omega}\mathbf{A}\nabla G(\cdot, y) \cdot \nabla\phi+\vec b \cdot \nabla\phi \,G(\cdot, y)+\vec c \cdot \nabla G(\cdot, y) \,\phi+dG(\cdot, y)\phi=\phi(y)
\]
for all $\phi\in C_c^{\infty}(\Omega)$.
However, there is an issue of integrability of the term $d G(\cdot,y) \phi$ in the above because $G(\cdot,y)$ is not expected to be a member of $L^{\frac{n}{n-2}}(\Omega)$; it only belongs to weak $L^{\frac{n}{n-2}}(\Omega)$, or the Lorentz space $L^{\frac{n}{n-2},\infty}(\Omega)$.
For this reason, we will defer the definition of Green's function (see Definition \ref{GreenDefinition})  after we have established further properties of solutions of $Lu=F$ in Section~\ref{sec3}.

\subsection{Two quantities}
For $f \in L^n(\Omega)$ and given $t>0$, it will be useful to split
\begin{equation}			\label{eq:ft}
\abs{f}=\abs{f}_{(t)}+\abs{f}^{(t)},
\end{equation}
where $\abs{f}_{(t)}=\abs{f}$ when $\abs{f}<t$ and $\abs{f}_{(t)}=0$ otherwise.
We also define, for a function $f\in L^n(\Omega)$ and $t$, $\varepsilon>0$,
\begin{equation}\label{eq:R_f}
R_f(t)=\norm{\,\abs{f}^{(t)}}_n\quad \text{and}\quad  r_f(\varepsilon)=\inf \,\set{t>0 : R_f(t)<\varepsilon}.
\end{equation}
It is clear that $R_f(t)\to 0$ as $t\to +\infty$, for any $f\in L^n(\Omega)$; therefore, $r_f$ is well defined.
We then show the next lemma.

\begin{lemma}\label{UniformSplit}
Let $f\in L^n(\Omega)$, and $\set{f_m}$ be a sequence in $L^n(\Omega)$ that converges to $f$ in $L^n$.
For any $\varepsilon>0$, there exists a subsequence $\set{f_{k_m}}$ such that, for all $m\in\mathbb N$,
\[
r_{f_{k_m}}(\varepsilon)\le 2r_f\left(\frac{\varepsilon}{3}\right).
\]
\end{lemma}
\begin{proof}
Since $f_m\to f$ in $L^n$, a subsequence $f_{k_m}$ converges to $f$ a.e.
Let $t>0$ be such that $R_f(t)<\varepsilon/3$. 
Then, for almost every $x\in\Omega$ with $\abs{f(x)} \le t$,  $\abs{f_{k_m}(x)} \to \abs{f(x)}<2t$, hence $x\in \Set{\abs{f_{k_m}}<2t}$ eventually.
Therefore, for a.e. $x$ with $\abs{f(x)} \le  t$, we have $\chi_{\set{\abs{f_{k_m}}>2t}} \to 0$ pointwise. 
Hence, from the dominated convergence theorem, we obtain that
\[
\lim_{m\to \infty} \int_{ \abs{f} \le t} \abs{f}^n \,\chi_{\set{ \abs{f_{k_m}}>2t}} = 0.
\]
Hence,
\begin{align*}
\limsup_{m\to\infty}\int_{\abs{f_{k_m}}>2t} \abs{f}^n &\le \limsup_{m\to\infty}\left(\int_{\abs{f}\le t} \abs{f}^n\chi_{\set{\abs{f_{k_m}}>2t}}+\int_{\abs{f}>t} \abs{f}^n\chi_{\set{\abs{f_{k_m}}>2t}}\right)\\
&\le\limsup_{m\to\infty}\left(\int_{ \abs{f}\le t} \abs{f}^n\chi_{\set{\abs{f_{k_m}}>2t}}\right)+\int_{\abs{f}>t} \abs{f}^n\\
&=(R_f(t))^n<\frac{\varepsilon^n}{3^n},
\end{align*}
since $t$ is such that $R_f(t)<\varepsilon/3$.
Now note that
\[
\lim_{m\to \infty} \int_{\Omega}\Abs{\,\abs{f}^n- \abs{f_{k_m}}^n}=0,
\]
and thus there exists $m_0\in\mathbb N$ such that, for all $m\ge m_0$,
\[
\int_{\abs{f_{k_m}}>2t} \abs{f}^n<\limsup_{m\to\infty}\int_{\abs{f_{k_m}}>2t} \abs{f}^n+\frac{\varepsilon^n}{3^n}<\frac{2\varepsilon^n}{3^n},
\qquad\int_{\Omega} \Abs{\, \abs{f}^n-\abs{f_{k_m}}^n}<\frac{\varepsilon^n}{3^n}.
\]
Therefore, for any $m\ge m_0$,
\begin{align*}
\bignorm{\,\abs{f_{k_m}}^{(2t)}}_n^n&=\int_{ \abs{f_{k_m}}>2t} \abs{f_{k_m}}^n=\int_{\abs{f_{k_m}}>2t} \abs{f}^n+\int_{\abs{f_{k_m}}>2t}\left( \,\abs{f_{k_m}}^n-\abs{f}^n\right)\\
&\le\int_{\abs{f_{k_m}}>2t} \abs{f}^n+\int_{\Omega} \Abs{\,\abs{f}^n-\abs{f_{k_m}}^n} \le\frac{2\varepsilon^n}{3^n}+\frac{\varepsilon^n}{3^n}<\varepsilon^n.
\end{align*}
Hence, for any $t$ with $R_f(t)<\varepsilon/3$ and any $m\ge m_0$, we obtain that
\[
R_{f_{k_m}}(2t)<\varepsilon,
\]
and we have
\[
\set{t>0:R_f(t)<\varepsilon/3} \subseteq \set{t>0:R_{f_{k_m}}(2t)<\varepsilon}=\frac{1}{2}\, \set{2t>0:R_{f_{k_m}}(2t)<\varepsilon}.
\]
The last inclusion shows that
\[
\frac{1}{2}r_{f_{k_m}}(\varepsilon)\le r_f\left(\frac{\varepsilon}{3}\right),
\]
for all $m\ge m_0$, which completes the proof after relabeling the $f_{k_m}$.
\end{proof}

In the case that $f$ belongs to $L^p(\Omega)$, for $p>n$, we compute
\[
\Abs{\,\abs{f}>t}=\frac{1}{t^p}\int_{\abs{f}>t} t^p\,dx\le\frac{1}{t^p}\int_{\Omega} \abs{f}^p=\frac{\norm{f}_p^p}{t^p},
\]
therefore
\[
\int_{\Omega}\left(\abs{f}^{(t)}\right)^n=\int_{\abs{f}>t} \abs{f}^n\le\left(\int_{\Omega} \abs{f}^p\right)^{n/p} \Abs{\,\set{\abs{f}>t}}^{1-n/p}\le \norm{f}_p^n\, \norm{f}_p^{p-n}\,t^{n-p},
\]
which implies that
\begin{equation}\label{eq:ftBound}
\bignorm{\,\abs{f}^{(t)}}_n\le \norm{f}_p^{p/n}\,t^{1-p/n}.
\end{equation}
Hence, we are led to the next lemma.

\begin{lemma}\label{RfBound}
Let $f\in L^p(\Omega)$ for some $p>n$.
Then, for every $\varepsilon>0$, we have
\[
r_f(\varepsilon)\le \norm{f}_p^{\frac{p}{p-n}}\varepsilon^{-\frac{n}{p-n}}.
\]
\end{lemma}
\begin{proof}
Note that, if $s> \norm{f}_p^{\frac{p}{p-n}}\varepsilon^{-\frac{n}{p-n}}$, then \eqref{eq:ftBound} shows that
\[
R_f(s)=\norm{\,\abs{f}^{(s)}}_n\le \norm{f}_p^{p/n}s^{1-p/n}<\varepsilon,
\]
which implies that $s\in\set{t>0 : R_f(t)<\varepsilon}$.
This completes the proof.
\end{proof}

For $f\in L^n(\Omega)$ and $t$, $\varepsilon >0$, we also set
\begin{equation}		\label{eq:R_f2}
\tilde{R}_f(t)=\sup \,\Set{\norm{f}_{L^n(A)} : A\subset \Omega,\; \abs{A}<t},\qquad
\tilde{r}_f(\varepsilon)=\sup \Set{t>0: \tilde{R}_f(t)<\varepsilon}.
\end{equation}
Since $\tilde{R}_f(t)\to 0$ as $t\to 0^+$, the quantity $\tilde{r}_f$ is well defined for any $f\in L^n(\Omega)$.
The next lemma is the analog of Lemma~\ref{UniformSplit}, but for the quantity $\tilde{r}_f$.

\begin{lemma}\label{UniformSplit2}
Let $f\in L^n(\Omega)$, and $\set{f_m}$ be a sequence in $L^n(\Omega)$ that converges to $f$ in $L^n$.
For any $\varepsilon>0$, there exists a subsequence $\set{f_{k_m}}$ such that, for all $m\in\mathbb N$,
\[
\tilde{r}_{f_{k_m}}(\varepsilon)\ge\tilde{r}_f\left(\frac{\varepsilon}{3}\right).
\]
\end{lemma}
\begin{proof}
Let $t>0$ such that $\tilde{R}_f(t)<\varepsilon/3$, and let $A\subseteq\Omega$ be such that $\abs{A}<t$.
Then, we compute
\begin{align*}
\int_A \abs{f_m}^n&\le\int_A \abs{f}^n+\int_A \left( \abs{f_m}^n-\abs{f}^n\right)\le\int_A \abs{f}^n+\int_{\Omega} \Abs{\,\abs{f_m}^n-\abs{f}^n}\\
&\le\left(\tilde{R}_f(t)\right)^n+\int_{\Omega} \Abs{\,\abs{f_m}^n-\abs{f}^n} \le\frac{\varepsilon^n}{3^n}+\int_{\Omega} \Abs{\,\abs{f_m}^n-\abs{f}^n}.
\end{align*}
Since $f_m\to f$ in $L^n$, there exists $m_0\in\mathbb N$ such that, for all $m\ge m_0$,
\[
\int_{\Omega}\Abs{\,\abs{f_m}^n-\abs{f}^n}<\frac{\varepsilon^n}{3^n}.
\]
Hence, for all $m\ge m_0$, and all $A\subseteq\Omega$ with $\abs{A}<t$,
\[
\int_A \abs{f_m}^n\le\frac{\varepsilon^n}{3^n}+\int_{\Omega}\Abs{\,\abs{f_m}^n-\abs{f}^n}<\varepsilon^n.
\]
This shows that, if $\tilde{R}_f(t)<\varepsilon/3$, then, for all $m\ge m_0$, $\tilde{R}_{f_m}(t)<\varepsilon$; hence,
\[
\Set{t>0 : \tilde{R}_f(t)<\varepsilon/3}\subseteq \Set{t>0: \tilde{R}_{f_m}(t)<\varepsilon},
\]
which completes the proof after relabeling the sequence $\set{f_m}_{m\ge m_0}$.
\end{proof}

Moreover, similarly to Lemma~\ref{RfBound}, we can show the next lemma.

\begin{lemma}\label{RfBound2}
Let $f\in L^p(\Omega)$ for some $p>n$. Then, for every $\varepsilon>0$, we have
\[
\tilde{r}_f(\varepsilon)\ge\left(\varepsilon \norm{f}_p^{-1}\right)^{\frac{np}{p-n}}.
\]
\end{lemma}
\begin{proof}
Note that, from H{\"o}lder's inequality,
\[
\norm{f}_{L^n(A)} \le \norm{f}_{L^p(A)} \,\abs{A}^{\frac{1}{n}-\frac{1}{p}}\le \norm{f}_p \, \abs{A}^{\frac{1}{n}-\frac{1}{p}}.
\]
Hence, if $\abs{A}<\left(\varepsilon\norm{f}_p^{-1}\right)^{\frac{np}{p-n}}$, then $\norm{f}_{L^n(A)}\le \varepsilon$, which shows that
\[
\tilde{R}_f\left(\left(\varepsilon \norm{f}_p^{-1}\right)^{\frac{np}{p-n}}\right)<\varepsilon.
\]
This completes the proof.
\end{proof}

\section{A Priori estimates}			\label{sec3}

\subsection{Global Boundedness}
Our first step is to show a maximum principle for the inhomogeneous equation $Lu=F$.
For this reason, we first show the next lemma.

\begin{lemma}				\label{Pointwise}
Let $\Omega\subset \mathbb R^n$ be a domain with $\abs{\Omega}\le 1$.
Let $\mathbf{A}$ be bounded and satisfy the uniform ellipticity condition \eqref{ellipticity} and $\vec b$, $\vec c\in L^n(\Omega)$, $d\in L^{n/2}(\Omega)$, and  $d \ge \dv \vec b$.
Let also $\vec{f}\in L^q(\Omega)$, $g\in L^{q/2}(\Omega)$ for some $q>n$.
Then if $u$ is a $W^{1,2}(\Omega)$ subsolution of 
\[
Lu=-\dv(\mathbf{A}\nabla u+\vec bu)+\vec c \cdot \nabla u+du= g+\dv\vec{f} \quad\text{in} \quad \Omega
\]
satisfying $u \le 0$ on $\partial\Omega$, we have
\[
\sup_{\Omega} u\le C\left( \norm{u^+}_2+\norm{\vec f}_q+\norm{g}_{q/2}\right),
\]
where $C$ depends only on $n$, $\lambda$, $q$, and $r_{\vec b-\vec c}\left(\frac{\lambda}{3C_n}\right)$.
\end{lemma}

\begin{proof}
Set $k=\norm{g}_{q/2}+\norm{\vec f}_q$.
We use the same test functions as in the proof of \cite[Theorem 8.15]{Gilbarg}.
For $\beta \ge 1$ and $N>k$, let us define
\[
H(z)=\left\{
\begin{array}{c l}
z^{\beta}-k^{\beta}, & k\le z\le N\\ \beta N^{\beta-1}z+(1-\beta)N^{\beta}-k^{\beta}, & z>N.
\end{array}\right.
\]
Note that $H\in C^1[k,\infty)$ and $H'$ is bounded.
We also set
\[
G(s)=\int_k^s \abs{H'(z)}^2\,dz \quad\text{for }\; s\ge k.
\]
Simple computations will show that
\begin{equation}\label{eq:G'}
G(s)\le H(s)H'(s)\quad\text{and}\quad H(s)\le sH'(s)\quad\text{for }\; s\ge k.
\end{equation}
We now define $w=u^{+} +k$ and $v=G(w)$.
Then, $\nabla v=G'(w)\nabla w=\abs{H'(w)}^2\,\nabla w$.
By using $v$ as a test function, we obtain
\[
\int_{\Omega}\mathbf{A}\nabla u \cdot \nabla v+\vec b \cdot \nabla v\, u+\vec c \cdot \nabla u\, v+duv \le \int_{\Omega}gv-\vec{f}\cdot\nabla v.
\] 
Note that we have $uv\ge 0$ from the definition of $v$.
Hence, since $d\ge\dv\vec b$, we write
\[
\int_{\Omega}\vec b \cdot \nabla v \, u+duv=\int_{\Omega}\vec b \cdot\nabla(uv)-\vec b\cdot\nabla u \, v+duv\ge-\int_{\Omega}\vec b \cdot\nabla u \, v,
\]
which implies that
\begin{equation}\label{eq:MainInLemma}
\int_{\Omega}\mathbf{A}\nabla u \cdot \nabla v\le\int_{\Omega} gv- \vec f \cdot \nabla v+(\vec b-\vec c) \cdot\nabla u  \,v.
\end{equation}
We now compute
\begin{equation}\label{eq:ANablaH}
\int_{\Omega}\mathbf{A}\nabla u \cdot\nabla v=\int_{\Omega}\mathbf{A}\nabla w \cdot\nabla v=\int_{\Omega}\mathbf{A}\nabla w\cdot H'(w)^2\,\nabla w=\int_{\Omega}\mathbf{A}\nabla H(w) \cdot \nabla H(w).
\end{equation}
Set $\hat q=\frac{2q}{q-2}$, and note that $\frac{1}{2}+\frac{1}{q}+\frac{1}{\hat q}=1$. Hence, using the generalized H{\"o}lder's inequality and Cauchy's inequality, we compute
\begin{align*}
\int_{\Omega}\abs{\vec f \cdot\nabla v} &\le \frac{1}{k} \int_{\Omega}\abs{\vec f}\, \abs{w\nabla v}= \frac{1}{k} \int_{\Omega} \abs{\vec f}\, \abs{H'(w)w\, H'(w)\nabla w}\\
&\le\ \frac{1}{k} \left(\int_{\Omega} \abs{\vec f}^q\right)^{1/q}\left(\int_{\Omega} \abs{wH'(w)}^{\hat q}\right)^{1/{\hat q}}\left(\int_{\Omega} \abs{\nabla H(w)}^2\right)^{1/2}\\
&\le\frac{1}{4\delta}\left(\int_{\Omega} \abs{wH'(w)}^{\hat q}\right)^{2/{\hat q}}+\delta\int_{\Omega} \abs{\nabla H(w)}^2.
\end{align*}
Therefore, if we choose $\delta=\frac{\lambda}{4}$, we obtain that
\begin{equation}\label{eq:FV}
\int_{\Omega} \abs{\vec{f}\cdot\nabla v} \le\frac{1}{\lambda}\left(\int_{\Omega} \abs{wH'(w)}^{\hat q}\right)^{2/{\hat q}}+\frac{\lambda}{4}\int_{\Omega} \abs{\nabla H(w)}^2.
\end{equation}
Using \eqref{eq:G'}, we get $G(s) \le s G'(s)$ and thus, since $\frac{2}{\hat{q}}+\frac{2}{q}=1$, we have
\begin{align}\label{eq:fv}
\nonumber
\int_{\Omega} \abs{gv}& \le \frac{1}{k} \int_{\Omega} \abs{g} \,wG(w)\le \frac{1}{k} \int_{\Omega} \abs{g} \, w^2G'(w)= \frac{1}{k} \int_{\Omega}\abs{g}\, \abs{H'(w)w}^2\\
&\le \frac{1}{k} \left(\int_{\Omega}\abs{g}^{\frac{q}{2}}\right)^{\frac{2}{q}}\left(\int_{\Omega} \abs{wH'(w)}^{\hat q}\right)^{\frac{2}{\hat q}} \le\left(\int_{\Omega} \abs{wH'(w)}^{\hat q} \right)^{\frac{2}{\hat q}}.
\end{align}
We now turn to bounding the term involving $\vec b-\vec c$.
For this reason, we split
\[
\abs{\vec b - \vec c}=\abs{\vec b- \vec c}_{(t)}+ \abs{\vec b - \vec c}^{(t)}
\]
as in \eqref{eq:ft}.
From the definition of $v$, we have $\nabla u\,v \equiv \nabla w\,G(w)$.
By using \eqref{eq:G'} and Cauchy's inequality with $\delta>0$, we then obtain
\begin{align}
						\nonumber
\int_{\Omega}(\vec b-\vec c) \cdot\nabla u \, v&=\int_{\Omega}(\vec b - \vec c) \cdot \nabla w \,G(w)\le\int_{\Omega} \left( \abs{\vec b-\vec c}_{(t)}+\abs{\vec b-\vec c}^{(t)} \right) \abs{H'(w)\nabla w} \, \abs{H(w)}\\
						\nonumber
&\le \norm{\,\abs{\vec b-\vec c}^{(t)}}_n \norm{\nabla H(w)}_2 \norm{H(w)}_{2^*}+t \norm{H(w)}_2 \norm{\nabla H(w)}_2\\
						\nonumber
&\le  C_nR_{\vec b-\vec c}(t) \,\norm{\nabla H(w)}_2^2+ t\delta \norm{\nabla H(w)}_2^2+\frac{t}{4\delta} \norm{H(w)}_2^2\\
						\label{eq:b-cBound}
&\le \left(C_nR_{\vec b-\vec c}(t)+t\delta\right)\int_{\Omega} \abs{\nabla H(w)}^2+\frac{t}{4\delta}\int_{\Omega} \abs{H(w)}^2.
\end{align}
We choose $t$ and $\delta$ so that $R_{\vec b-\vec c}(t)<\frac{\lambda}{3 C_n}$ and $\delta=\frac{\lambda}{6t}$.
Then, by using \eqref{eq:G'} and H\"older's inequality (since $q^*>2$), we obtain that (recall $\abs{\Omega} \le 1$)
\begin{align*}
\int_{\Omega}(\vec b-\vec c)\cdot \nabla u \,v&\le\frac{\lambda}{2}\int_{\Omega} \abs{\nabla H(w)}^2+\frac{3t^2}{2\lambda}\int_{\Omega} \abs{H(w)}^2\\
&\le\frac{\lambda}{2}\int_{\Omega} \abs{\nabla H(w)}^2+\frac{3t^2}{2\lambda}\left(\int_{\Omega} \abs{wH'(w)}^{\hat q}\right)^{2/{\hat q}}. 
\end{align*}
Since this inequality holds for all $t$ with $R_{\vec b-\vec c}(t)<\frac{\lambda}{3C_n}$, it holds for $t=r_{\vec b-\vec c}\left(\frac{\lambda}{3C_n}\right)$ as well.
Hence, 
\begin{equation} \label{eq:Forb-c}
\int_{\Omega}(\vec b-\vec c) \cdot \nabla u \, v\le\frac{\lambda}{2}\int_{\Omega} \abs{\nabla H(w)}^2+\frac{3 r_{\vec b-\vec c}\left(\frac{\lambda}{3C_n}\right)^2}{2\lambda} \left(\int_{\Omega} \abs{wH'(w)}^{\hat q}\right)^{2/{\hat q}}. 
\end{equation}

We now substitute \eqref{eq:ANablaH}, \eqref{eq:FV}, \eqref{eq:fv}, and \eqref{eq:Forb-c} in \eqref{eq:MainInLemma} and then use the ellipticity condition \eqref{ellipticity} and the Sobolev inequality to obtain that
\[
\norm{H(w)}_{2^*} \le C \norm{w H'(w)}_{\hat q},
\]
where $C$ depends on $n$, $\lambda$, and $r_{\vec b-\vec c}\left(\frac{\lambda}{3C_n}\right)$.
The above estimate corresponds to \cite[(8.36), p. 190]{Gilbarg}.
Continuing as in \cite{Gilbarg}, we obtain
\[
\sup_{\Omega}u\le\sup_{\Omega}u^+\le\sup_{\Omega}w \le C \norm{w}_2  \le C  \left(\norm{u^+}_2+k \right),
\]
where we used that $\abs{\Omega} \le 1$ in the last step.
This completes the proof.
\end{proof}

\begin{remark}				\label{rmk1344sat}
If $u$ is  a (sub)solution of the equation
\[
-\dv (\mathbf{A} \nabla u+ \vec b u)+\vec c \cdot \nabla u + d u =f + \dv\vec g\quad\text{in}\quad\Omega, 
\]
then for any $s>0$, the function $u_s(x)=u(sx)$ is a (sub)solution of 
\[
-\dv (\mathbf{A}_s \nabla u_s+ s\vec b_s u_s)+s\vec c_s \cdot \nabla u_s + s^2 d_s u_s =s^2f_s + \dv s\vec g_s\quad\text{in}\quad \Omega_s, 
\]
where $\Omega_s=\frac{1}{s}\Omega$ and $\mathbf{A}_s(x)=\mathbf{A}(sx)$, etc.
Note that $d \ge \dv \vec b$ in $\Omega$ if and only if $s^2d_s \ge s \dv \vec b_s$ in $\Omega_s$.
Also, from \eqref{eq:R_f}, we find $R_{f_s}(t)=s^{-1} R_f(t)$, $R_{sf}(t)=sR_f(t/s)$, and thus, 
\[r_{sf_s}(\varepsilon)=s r_{f}(\varepsilon).\]
Therefore, the conclusion of Lemma~\ref{Pointwise} remains the same for any domain $\Omega$ with $\abs{\Omega}<+\infty$ except that the constant $C$ now depends additionally on $\abs{\Omega}$.
The same comment applies to Proposition~\ref{BoundForCriticalSubsolutions} below.
\end{remark}

We now show the maximum principle.

\begin{proposition}\label{BoundForCriticalSubsolutions}
Let $\Omega\subset \mathbb R^n$ be a domain with $\abs{\Omega}\le 1$.
Let $\mathbf{A}$ be bounded and satisfy the uniform ellipticity condition \eqref{ellipticity} and $\vec b$, $\vec c \in L^n(\Omega)$, $d\in L^{n/2}(\Omega)$, and $d\ge\dv\vec b$.
Let also $\vec f\in L^q(\Omega)$, $g\in L^{q/2}(\Omega)$ for $q>n$.
Then if  $u$ is a  $W^{1,2}(\Omega)$ subsolution of 
\[
Lu=-\dv(\mathbf{A}\nabla u+\vec bu)+\vec c\cdot \nabla u+du = g+\dv\vec{f} \quad\text{in}\quad \Omega,
\]
we have
\[
\sup_{\Omega}u\le \sup_{\partial\Omega} u^{+} +C\left( \norm{\vec f}_q+\norm{g}_{q/2} \right),
\]
where $C$ depends only on $n$, $\lambda$, $q$, and $r_{\vec b-\vec c}\left(\frac{\lambda}{3C_n}\right)$.
\end{proposition}
\begin{proof}
We follow the proof of \cite[Theorem 8.16]{Gilbarg}.
We assume without loss of generality that $\sup_{\partial \Omega} u^{+} =0$. 
Set $k=\norm{\vec f}_q+\norm{g}_{q/2}$ and assume that $k>0$.
First, note that from the proof of  Lemma~\ref{Pointwise},  we have $M:=\sup_{\Omega}u^+<+\infty$. 
We now follow the proof of \cite[Theorem 8.16]{Gilbarg} and using $\abs{\Omega} \le 1$ to obtain
\[
\int_{\Omega} \abs{\nabla w}^2\le
C \left( \int_{\Omega} \abs{\vec b-\vec c}^2+1 \right),
\quad\text{where }\; w:=\ln\dfrac{M+k}{M+k-u^+}\;\text{ and }\; C=C(\lambda).
\]
Splitting $\vec b-\vec c$ as in Lemma~\ref{Pointwise}, we obtain that, for any $t>0$,
\[
\norm{\vec b-\vec c}_2 \le \norm{\,\abs{\vec b-\vec c}_{(t)}}_2+\norm{\,\abs{\vec b-\vec c}^{(t)}}_2 \le t+\norm{\,\abs{\vec b -\vec c}^{(t)}}_n \le t+R_{\vec b -\vec c}(t).
\]
Therefore, letting $t\to r_{\vec b -\vec c} \left(\frac{\lambda}{3 Cn}\right)$ from above, we obtain that
\[
\int_{\Omega} \abs{\vec b-\vec c}^2\le \textstyle \left(r_{\vec b-\vec c}\left(\frac{\lambda}{3C_n}\right)+\frac{\lambda}{3C_n} \right)^2,
\]
and thus by the Sobolev inequality, we have (recall $\abs{\Omega} \le 1$)
\begin{equation}\label{eq:Normw2}
\norm{w}_2\le C_n \norm{\nabla w}_2 \le \textstyle C(n, \lambda) \left( r_{\vec b-\vec c}\left(\frac{\lambda}{3C_n}\right)+\frac{\lambda}{3C_n} +1 \right).
\end{equation}
Continuing as in \cite[Theorem 8.16]{Gilbarg}, we obtain that $w$ is a subsolution satisfying
\[
-\dv(\mathbf{A}\nabla w)+(\vec c-\vec b)\cdot\nabla w\le\frac{\abs{g}}{k}+\frac{\abs{\vec f}^2}{2\lambda k^2}+\dv\left(\frac{\vec{f}}{M+k-u^+}\right).
\]
Therefore, by Lemma~\ref{Pointwise} and \eqref{eq:Normw2}, we have
\[
\sup_{\Omega}w\le C\left(\norm{w}_2+\Norm{\frac{\abs{g}}{k}+\frac{\abs{\vec f}^2}{2\lambda k^2}}_{\frac{q}{2}}+\Norm{\frac{\vec f}{M+k-u^+}}_q\right)\le C \norm{w}_2 + C \le C,
\]
where $C$ depends only on $n$, $\lambda$, $q$, and $r_{\vec b-\vec c}\left(\frac{\lambda}{3C_n}\right)$.
The last inequality then shows that $\ln\frac{M+k}{k}\le C$ and the proof is complete.
\end{proof}

\subsection{Local Boundedness}
In the following, we will need a local analog of Lemma~\ref{Pointwise} for solutions to the equation $L\tran u=0$. We treat the subcritical case first.

\begin{lemma}\label{LocalBoundForp}
Let $\Omega \subset \bR^n$ be a domain with $\abs{\Omega} \le 1$.
Let $\mathbf{A}$ be bounded and satisfy the uniform ellipticity condition \eqref{ellipticity} and $\vec b$, $\vec c\in L^n(\Omega)$, $\vec b-\vec c \in L^p(\Omega)$ for $p>n$, $d\in L^{n/2}(\Omega)$, and $d\ge\dv\vec b$.
Suppose that $u\in W^{1,2}(\Omega)$ is a nonnegative solution to the equation
\[
L\tran u=-\dv(\mathbf{A}\tran\nabla u+\vec cu)+\vec b\cdot\nabla u+du=0 \quad\text{in}\quad \Omega.
\]
Then, for any ball $B_r \subset \subset \Omega$, we have
\[
\sup_{B_{r/2}}u\le C\fint_{B_r}u,
\]
where $C$ depends on $n$, $\lambda$, $p$, $\norm{\mathbf A}_{\infty}$, and $\norm{\vec b-\vec c}_p$.
\end{lemma}
\begin{proof}
Note first that, for any $\phi\in C_c^{\infty}(\Omega)$ with $\phi\ge 0$,
\[
\int_{\Omega}\mathbf{A}\tran\nabla u \cdot\nabla\phi+(\vec c-\vec b) \cdot \nabla\phi\, u\le\int_{\Omega}\mathbf{A}\tran\nabla u \cdot\nabla\phi+\vec c \cdot\nabla\phi\, u+\vec b\cdot \nabla u \, \phi+du\phi=0,
\]
since $u\phi\ge 0$ and $d\ge\dv\vec b$ in $\Omega$.
Therefore, $u$ is a subsolution of
\[
-\dv(\mathbf{A}\tran\nabla u+(\vec c-\vec b)u)= 0 \quad\text{in}\quad\Omega.
\]
Since $\abs{\Omega} \le 1$, we must have $r \lesssim 1$.
Then, we follow the proof of in \cite[Theorem 8.17]{Gilbarg} (see also \cite[p. 199, Theorem 13.1]{LU68}) to get
\[
\sup_{B_{r/2}}u\le C\left(\fint_{B_r} \abs{u}^2\right)^{\frac12}
\]
for some constant $C$ that depends only on $n$, $\lambda$, $p$, $\norm{\mathbf A}_\infty$, and $\norm{\vec b-\vec c}_p$.
To obtain the desired $L^1$ bound, we apply a standard procedure as in \cite[pp. 80--82]{Gi93}.
\end{proof}

As we shall show in Proposition~\ref{ToCounterexample}, the previous lemma does not hold when $\vec b - \vec c \in L^n(\Omega)$, even in the case $d=0$, $\vec b=0$, and $\mathbf{A}=\mathbf{I}$.
In order to obtain an analogue of Lemma~\ref{LocalBoundForp} for the case when $\vec b -\vec c \in L^n$, we shall assume that $d\ge \dv \vec c$.

\begin{lemma}\label{LocalBoundForn}
Let $\Omega \subset \bR^n$ be a domain with $\abs{\Omega} \le 1$.
Let $\mathbf{A}$ be bounded and satisfy the uniform ellipticity condition \eqref{ellipticity} and $\vec b,\vec c\in L^n(\Omega)$, $d\in L^{n/2}(\Omega)$, and $d\ge\dv\vec c$.
Suppose also that $u\in W^{1,2}(\Omega)$ is a nonnegative solution to the equation
\[
L\tran u=-\dv(\mathbf{A}\tran\nabla u+\vec cu)+\vec b\cdot \nabla u+du=0\quad\text{in}\quad \Omega.
\]
Then, for any ball $B_r$ that is compactly supported in $\Omega$, we have
\[
\sup_{B_{r/2}}u\le C\fint_{B_r}u,
\]
where $C$ depends on $n$, $\lambda$, $\norm{\mathbf A}_{\infty}$, and $r_{\vec b-\vec c}\left(\frac{\lambda}{3C_n}\right)$.
\end{lemma}
\begin{proof}
With a procedure similar to the proof of Lemma \ref{LocalBoundForp}, we first show that $u$ is a subsolution of
\[
-\dv(\mathbf{A}\tran\nabla u)+(\vec b-\vec c) \cdot \nabla u= 0\quad\text{in}\quad \Omega.
\]
The proof then follows the lines of the proof of \cite[Theorem 8.17]{Gilbarg}, where we apply a procedure similar to \eqref{eq:b-cBound}, and we obtain that
\[
\sup_{B_{r/2}}u\le C\left(\fint_{B_r} \abs{u}^2\right)^{1/2},
\]
for some constant $C>0$ that depends on $n$, $\lambda$, $\norm{\mathbf A}_\infty$, and $r_{\vec b-\vec c}\left(\frac{\lambda}{3C_n}\right)$.
We then apply a standard procedure as in \cite[pp. 80--82]{Gi93} to complete the proof.
\end{proof}

\subsection{Lorentz Regularity}
In this subsection we show that solutions that belong to $L^{\frac{n}{n-2},\infty}(\Omega)$ have gradients that belong to $L^{\frac{n}{n-1},\infty}(\Omega)$.
To show this, we first need the next variant of Caccioppoli's estimate.

\begin{lemma}			\label{Caccioppoli}
Let $\mathbf{A}$ be bounded and satisfy the uniform ellipticity condition \eqref{ellipticity} and $\vec b$, $\vec c \in L^n(\Omega)$, $d\in L^{n/2}(\Omega)$, and $d\ge\dv\vec b$ (or $d\ge\dv\vec c$).
Let also $B_s=B_s(x)\subset \Omega$ be a ball compactly supported in $\Omega$ with $s<\frac13$.
Let $\Omega'=\Omega\setminus B_s$ and suppose that $u\in W^{1,2}(\Omega')$, $u=0$ on $\partial\Omega$, and $u$ satisfies the equation
\[
-\dv(\mathbf{A}\nabla u+\vec bu)+\vec c\cdot\nabla u+du=0\quad\text{in}\quad \Omega'.
\]
Then, for any $r\in(3s,2]$, we have
\[
\int_{\Omega\cap (B_r\setminus B_{r/2})}  \,\abs{\nabla u}^2\le \frac{C}{r^2}\int_{\Omega\cap (B_{2r} \setminus B_{r/3})} \, \abs{u}^2.
\]
Moreover, if $r\ge 1$, we have
\[
\int_{\Omega\setminus B_{r/2}} \, \abs{\nabla u}^2\le C\int_{\Omega\setminus B_{r/3}}\, \abs{u}^2.
\]
In the above, $C$ depends only on $n$, $\lambda$, $\norm{\mathbf A}_{\infty}$, and $r_{\vec b-\vec c}\left(\frac{\lambda}{3C_n}\right)$.
\end{lemma}
\begin{proof}
Consider first the case when $d\ge\dv\vec b$.
Let $\phi$ be a smooth and nonnegative cutoff function which vanishes in $B_s$.
Then, we have $u\phi^2\in W_0^{1,2}(\Omega')$, and thus, using $u\phi^2$ as a test function, we obtain that
\[
\int_{\Omega'}\mathbf{A}\nabla u\cdot\nabla(u\phi^2)+\vec b\cdot\nabla(u\phi^2)\,u+\vec c\cdot\nabla u\,u\phi^2+du^2\phi^2=0.
\]
Since $d\geq\dv\vec b$ and $u^2\phi^2\geq 0$, we obtain that
\[
\int_{\Omega}\vec b\cdot\nabla(u\phi^2)\,u+du^2\phi^2=\int_{\Omega}\vec b\cdot\nabla(u^2\phi^2)+du^2\phi^2-\int_{\Omega}\vec b\cdot\nabla u\,u\phi^2\geq -\int_{\Omega}\vec b\cdot\nabla u\,u\phi^2,
\]
which implies that
\[
\int_{\Omega'}\mathbf{A}\nabla u\cdot\nabla(u\phi^2)+(\vec c-\vec b)\cdot\nabla u\, u\phi^2 \le 0.
\]
Using the ellipticity of $\mathbf{A}$, this implies that
\begin{align}
				\nonumber
\lambda\int_{\Omega'} \abs{\phi\nabla u}^2 &\le \int_{\Omega'} \mathbf{A} \nabla u\cdot\nabla u\,\phi^2 = \int_{\Omega'} \mathbf{A} \nabla u \cdot \nabla(u \phi^2) - 2\mathbf{A} \nabla u\cdot \nabla \phi\, u\phi\\
				\label{eq:I_1+I_2}
&\le \int_{\Omega'} (\vec b-\vec c)\cdot\nabla u\,u\phi^2-\int_{\Omega'}2\mathbf{A}\nabla u\cdot\nabla\phi\,u\phi:=I_1+I_2.
\end{align}
To bound $I_2$, we use boundedness of $\mathbf A$ to deduce that
\begin{equation}\label{eq:I_2Bound}
I_2 \le 2\norm{\mathbf A}_\infty\, \norm{\phi\nabla u}_2 \,\norm{u\nabla\phi}_2
\le \frac{\lambda}{12} \norm{\phi\nabla u}_2^2+\frac{12\norm{\mathbf A}_\infty^2}{\lambda} \norm{u\nabla\phi}_2^2.
\end{equation}
In order to bound $I_1$, we split $\abs{\vec b-\vec c}$ as in \eqref{eq:ft}, and we obtain that, for any $t>0$,
\begin{align*}
I_1&\le t\, \norm{\phi\nabla u}_2 \,\norm{\phi u}_2+R_{\vec b-\vec c}(t) \, \norm{\phi\nabla u}_2 \, \norm{\phi u}_{2^*}\\
&\le t \,\norm{\phi\nabla u}_2 \, \norm{\phi u}_2+C_n R_{\vec b-\vec c}(t) \, \norm{\phi\nabla u}_2 \, \norm{\nabla(\phi u)}_2\\
&\le t \,\norm{\phi\nabla u}_2 \, \norm{\phi u}_2+C_nR_{\vec b-\vec c}(t) \,\norm{\phi\nabla u}_2^2+C_nR_{\vec b-\vec c}(t) \,\norm{\phi\nabla u}_2 \, \norm{u\nabla\phi}_2\\
&\le t \,\norm{\phi\nabla u}_2 \, \norm{\phi u}_2 + 2C_nR_{\vec b-\vec c}(t)\, \norm{\phi\nabla u}_2^2+C_nR_{\vec b-\vec c}(t) \, \norm{u\nabla\phi}_2^2\\
&\le \left( t\delta+2C_nR_{\vec b-\vec c}(t)\right)\, \norm{\phi\nabla u}_2^2+\frac{t}{4\delta} \,\norm{\phi u}_2^2+C_nR_{\vec b-\vec c}(t)\, \norm{u\nabla\phi}_2^2.
\end{align*}
Choosing $t>r_{\vec b-\vec c}\left(\frac{\lambda}{3C_n}\right)$ and $\delta=\frac{\lambda}{6t}$, we obtain that
\begin{equation}\label{eq:I_1Bound}
I_1\le\frac{5\lambda}{6}  \norm{\phi\nabla u}_2^2+\frac{3t^2}{2\lambda} \norm{\phi u}_2^2+\frac{\lambda}{3} \norm{u\nabla\phi}_2^2.
\end{equation}
Substituting \eqref{eq:I_2Bound} and \eqref{eq:I_1Bound} in \eqref{eq:I_1+I_2}, we obtain that
\[
\lambda \norm{\phi\nabla u}_2^2 \le \frac{11}{12} \lambda \norm{\phi\nabla u}_2^2+\frac{3t^2}{2\lambda} \norm{\phi u}_2^2+\left(\frac{\lambda}{3}+\frac{12\norm{\mathbf A}_\infty^2}{\lambda}\right) \norm{u\nabla\phi}_2^2
\]
for any $t>0$ with $t>r_{\vec b-\vec c}\left(\frac{\lambda}{3C_n}\right)$.
Hence, by letting $t\to r_{\vec b-\vec c}\left(\frac{\lambda}{3C_n}\right)$, we obtain
\begin{equation}\label{eq:In2^*Bound}
\norm{\phi\nabla u}_2^2 \le C\norm{\phi u}_2^2+C\norm{u\nabla\phi}_2^2,
\end{equation}
where $C$ depends only on $n$, $\lambda$, $\norm{\mathbf A}_{\infty}$, and $r_{\vec b-\vec c}\left(\frac{\lambda}{3C_n}\right)$.

We now choose $\phi$ such that $0\le\phi\le 1$, $\phi \equiv 1$ in $B_r \setminus B_{r/2}$, $\phi$ is supported in $B_{2r} \setminus B_{r/3}$, and $\abs{\nabla\phi} \le C r^{-1}$.
Then $\phi$ vanishes in $B_s$, and thus \eqref{eq:In2^*Bound} shows the first inequality.
To show the second inequality, we choose $\phi$ such that $0\le\phi\le 1$, $\phi$ is supported in $\mathbb R^n\setminus B_{r/3}$, $\phi\equiv 1$ in $\mathbb R^n\setminus B_{r/2}$, and $\abs{\nabla\phi} \le Cr^{-1}$.

Now, consider the case when $d\ge\dv\vec c$.
In this case, $u$ satisfies the inequality
\[
\int_{\Omega'}\mathbf{A}\nabla u\cdot\nabla(u\phi^2)+(\vec b-\vec c)\cdot\nabla(u\phi^2)\, u \le 0,
\]
and the only change in \eqref{eq:I_1+I_2} is the term $I_1$, which is equal to
\[
I_1=\int_{\Omega'}(\vec c-\vec b)\cdot\nabla(u\phi^2)\,u=\int_{\Omega'}(\vec c-\vec b)\cdot\nabla u\,u\phi^2+2\int_{\Omega'}(\vec c-\vec b)\cdot\nabla \phi\,u^2\phi=I_3+I_4.
\]
Then, similar to \eqref{eq:I_1Bound}, we have
\begin{equation*}	
I_3\le\frac{5\lambda}{6} \norm{\phi\nabla u}_2^2+\frac{3t^2}{2\lambda} \norm{\phi u}_2^2+\frac{\lambda}{3} \norm{u\nabla\phi}_2^2,
\end{equation*}
for any $t>r_{\vec b-\vec c}\left(\frac{\lambda}{3C_n}\right)$.
To bound $I_4$, we split again $\vec b-\vec c$ as in \eqref{eq:ft}, and we obtain
\begin{align*}
I_4& \le 2t \,\norm{u\nabla\phi}_2 \, \norm{u\phi}_2+2R_{\vec b-\vec c}(t)\, \norm{u\nabla\phi}_2 \,\norm{u\phi}_{2^*}\\
&\le 2t \,\norm{u\nabla\phi}_2 \, \norm{u\phi}_2+2C_nR_{\vec b-\vec c}(t)\,\norm{u\nabla\phi}_2^2+2C_nR_{\vec b-\vec c}(t)\,\norm{u\nabla\phi}_2\,\norm{\phi\nabla u}_2\\
&\le \left(t+2C_nR_{\vec b-\vec c}(t)+\frac{C_nR_{\vec b-\vec c}(t)}{2\delta}\right) \norm{u\nabla\phi}_2^2+t\,\norm{u\phi}_2^2+2C_nR_{\vec b-\vec c}(t) \,\delta\,\norm{\phi\nabla u}_2^2
\end{align*}
for any $t>r_{\vec b-\vec c}\left(\frac{\lambda}{3C_n}\right)$ and any $\delta>0$. We now let $t\to r_{\vec b-\vec c}\left(\frac{\lambda}{3C_n}\right)$ and choose $\delta$ sufficiently small, and we proceed as above to finish the proof.
\end{proof}

We will also need the next reverse H{\"o}lder inequality for solutions.

\begin{lemma}\label{2By1}
Assume the same hypotheses as in Lemma~\ref{Caccioppoli}.
Then, for any $r\in (4s,2]$, we have
\[
\int_{\Omega\cap(B_{2r} \setminus B_{r/3})} \abs{u}^2\le\frac{C}{r^n}\left(\int_{\Omega \cap(B_{3r}\setminus B_{r/4})} \abs{u}\right)^2,
\]
where $C$ depends only on $n$, $\lambda$, $\norm{\mathbf A}_{\infty}$, and $r_{\vec b-\vec c}\left(\frac{\lambda}{3C_n}\right)$.
Moreover, we also have
\[
\int_{\Omega\setminus B_{2/3}}\abs{u}^2\le C\left(\int_{\Omega\setminus B_{1/3}} \abs{u}\right)^2.
\]
\end{lemma}
\begin{proof}
Note that in both cases $d\ge\dv\vec b$ and $d\ge\dv\vec c$, the proof of Lemma~\ref{Caccioppoli} shows that the inequality \eqref{eq:In2^*Bound} holds for any smooth nonnegative $\phi$ that vanishes in $B_s$. 
We now add the term $\norm{u\nabla\phi}_2^2$ in both sides of \eqref{eq:In2^*Bound}.
Since $u\phi\in W_0^{1,2}(\Omega')$, using the Sobolev inequality we obtain that
\[
\norm{u\phi}_{2^*}^2\le C_n^2 \norm{\nabla(u\phi)}_2^2 \le C\norm{\phi u}_2^2+C\norm{u\nabla\phi}_2^2,
\]
where $C$ depends only on $n$, $\lambda$, $\norm{\mathbf A}_{\infty}$, and $r_{\vec b-\vec c}\left(\frac{\lambda}{3C_n}\right)$.

To show the first estimate, fix any $y \in \Omega$ such that $B_{r/12}(y)\cap B_s=\emptyset$.
Let $0< t<\tau< r/12$ and choose $\phi$ such that $0\le\phi\le 1$, $\phi$ is supported in $B_{\tau}(y)$, $\phi\equiv 1$ in $B_t(y)$, and $\abs{\nabla\phi} \le C (\tau-t)^{-1}$.
Then, we get
\begin{equation}\label{eq:Overt^2}
\left(\int_{E_t} \abs{u}^{2^*}\right)^{2/2^*} \le \frac{C}{(\tau -t)^2}\int_{E_{\tau}} \abs{u}^2,\quad\text{where }\; E_t:=B_t(y) \cap \Omega.
\end{equation}
We now use H{\"o}lder's inequality, to obtain that
\begin{align*}
\int_{E_t} \abs{u}^2&=\int_{E_t} \abs{u}^{\frac{4}{n+2}}\, \abs{u}^{\frac{2n}{n+2}}
\le \left(\int_{E_t} \abs{u}\right)^{\frac{4}{n+2}}\left(\int_{E_t} \abs{u}^{\frac{2n}{n-2}}\right)^{\frac{n-2}{n+2}}
\le \left(\int_{E_t} \abs{u} \right)^{\frac{4}{n+2}}\left(\frac{C}{(\tau-t)^2}\int_{E_{\tau}} \abs{u}^2\right)^{\frac{n}{n+2}}\\
&=\left(\frac{C}{(\tau-t)^{\frac{n}{2}}} \int_{E_t} \abs{u} \right)^{\frac{4}{n+2}}\left(\int_{E_{\tau}} \abs{u}^2\right)^{\frac{n}{n+2}},
\end{align*}
where we used \eqref{eq:Overt^2} in the second inequality. We now use Young's inequality
\[
ab\le\frac{2}{n+2}\delta^{-\frac{n+2}{2}}a^\frac{n+2}{2}+\frac{n}{n+2}\delta^{\frac{n+2}{n}}b^{\frac{n+2}{n}}
\]
to bound the last term, and we obtain that
\[
\int_{E_t} \abs{u}^2\le\frac{2}{n+2}\delta^{-\frac{n+2}{2}} \,\frac{C}{(\tau-t)^n}\left(\int_{E_t} \abs{u}\right)^2+\frac{n}{n+2}\delta^{\frac{n+2}{n}}\int_{E_{\tau}} \abs{u}^2.
\]
Choosing $\delta>0$ such that $\frac{n}{n+2}\delta^{\frac{n+2}{n}}<\frac{1}{2}$, we finally obtain that
\[
\int_{E_t} \abs{u}^2\le\frac{C}{(\tau-t)^n}\left(\int_{E_t} \abs{u}\right)^2+\frac{1}{2}\int_{E_{\tau}} \abs{u}^2\le\frac{C}{(\tau-t)^n}\left(\int_{E_{r/12}} \abs{u}\right)^2+\frac{1}{2}\int_{E_{\tau}} \abs{u}^2.
\]
Hence, using Lemma 5.1 in \cite[p. 81]{Gi93}, we finally obtain that
\begin{equation}\label{eq:r/24}
\int_{E_{r/24}} \abs{u}^2 \le\frac{C}{r^n}\left(\int_{E_{r/12}} \abs{u}\right)^2.
\end{equation}
We now cover the annulus $B_{2r}\setminus B_{r/3}$ by $N$ balls of radius $r/24$ centered at points $y$ such that $B_{r/12}(y)\cap B_s=\emptyset$;
the number $N$ can be chosen so that it depends only on $n$.
Then, the union of the doubles of those balls is a subset of $B_{3r}\setminus B_{r/4}$.
Hence, adding the inequalities in \eqref{eq:r/24} for each one of such balls, we get the desired estimate.

To show the second estimate, we follow a similar procedure.
For any $t$, $\tau\in (0,\frac{1}{4})$ with $t<\tau$, we consider $\phi$ such that $0\leq\phi\leq 1$, $\phi=1$ outside $B_{1/3\, -t}$, $\phi=0$ in $B_{1/3\,-\tau}$, and $\abs{\nabla\phi} \le C(\tau-t)^{-1}$. Then, we obtain
\[
\left(\int_{F_t} \abs{u}^{2^*}\right)^{2/2^*} \le \frac{C}{(\tau -t)^2}\int_{F_{\tau}} \abs{u}^2,\quad\text{where }\; F_t:=\Omega\setminus B_{1/3\,-t}.
\]
As above, we use Young's inequality and H{\"o}lder's inequality, to obtain that
\[
\int_{F_t} \abs{u}^2\le\frac{C}{(\tau-t)^n}\left(\int_{\Omega\setminus B_{1/3}} \abs{u}\right)^2+\frac{1}{2}\int_{F_{\tau}} \abs{u}^2.
\]
Therefore, using Lemma 5.1 in \cite[p. 81]{Gi93}, we obtain that
\[
\int_{\Omega\setminus B_{2/3}}\abs{u}^2 \le C\left(\int_{\Omega\setminus B_{1/3}} \abs{u}\right)^2,
\]
which completes the proof.
\end{proof}

Finally, we show the next regularity estimate.

\begin{proposition}\label{GeneralLorentzBound}
Let $\Omega\subset \mathbb R^n$ be a domain with $\abs{\Omega}<+\infty$.
Suppose that $\mathbf{A}$ is bounded and satisfies the uniform ellipticity condition \eqref{ellipticity} and $\vec b ,\vec c \in L^n(\Omega)$, $d\in L^{n/2}(\Omega)$, with $d\ge\dv\vec b$, or $d\geq\dv\vec c$.
Let also $x_0\in\Omega$ and suppose that $u\in W_0^{1,q}(\Omega)$ for some $q \in[1, \frac{n}{n-1})$, $u\in W^{1,2}(\Omega\setminus B_s(x_0))$ for all $s>0$, and $u$ satisfies the equation
\[
\dv(\mathbf{A}\nabla u+\vec bu)+\vec c\cdot\nabla u+du=0\quad\text{in}\quad \Omega\setminus \set{x_0}.
\]
If $\norm{u}_{L^{\frac{n}{n-2},\infty}(\Omega)}<+\infty$, then we have
\[
\norm{\nabla u}_{L^{\frac{n}{n-1},\infty}(\Omega)} \le C \norm{u}_{L^{\frac{n}{n-2},\infty}(\Omega)}^2+C,
\]
where $C$ depends on $n$, $\lambda$, $\norm{\mathbf A}_{\infty}$, $r_{\vec b-\vec c}\left(\frac{\lambda}{3C_n}\right)$, and $\abs{\Omega}$.
\end{proposition}
\begin{proof}
Let
\[
\Omega_t=\set{ x \in \Omega: \abs{\nabla u(x)} \ge t}\quad\text{and}\quad K=\norm{u}_{L^{\frac{n}{n-2},\infty}(\Omega)}.
\]
We shall show $t^{\frac{n}{n-1}} \, \abs{\Omega_t} \le C(K^2+1)$ for all $t>0$.
Note first that, if $t\le 4^{n-1}$, then
\[
t^{\frac{n}{n-1}}\, \abs{\Omega_t} \le 4^{n}\abs{\Omega}.
\]
Therefore, it is enough to consider the case when $t>4^{n-1}$.
Let $t>4^{n-1}$ and set $\tau=t^{-\frac{1}{n-1}}$.
Then, we have
\[
\tau<\left(4^{n-1}\right)^{-\frac{1}{n-1}}=\frac{1}{4},\quad -\frac{\ln\tau}{\ln 2}>-\frac{\ln(1/4)}{\ln 2}=2.
\]
Now, let $r\in[4 \tau,2]$, write $B_r=B_r(x_0)$, and set
\[
D_r=\Omega\cap(B_r\setminus B_{r/2}),\quad E_r=\Omega\cap(B_{2r}\setminus B_{r/3}),\quad F_r=\Omega\cap(B_{3r}\setminus B_{r/4}).
\]
Then, from Lemmas \ref{Caccioppoli} and \ref{2By1} we obtain that
\[
\int_{D_r} \abs{\nabla u}^2\le\frac{C}{r^2}\int_{E_r} \abs{u}^2\le\frac{C}{r^{2+n}}\left(\int_{F_r} \abs{u}\right)^2,
\]
where $C$ depends on $n$, $\lambda$, $\norm{\mathbf{A}}_{\infty}$, and $r_{\vec b-\vec c}\left(\frac{\lambda}{3C_n}\right)$.
By H{\"o}lder's inequality for Lorentz norms, we have (see, for example, \cite[Theorem 1.4.17]{Grafakos})
\[
\int_{F_r} \abs{u} \le \norm{u}_{\frac{n}{n-2},\infty} \norm{\chi_{F_r}}_{\frac{n}{2},1} \le  C \norm{u}_{\frac{n}{n-2},\infty} \,\abs{F_r}^{2/n}\le C \norm{u}_{\frac{n}{n-2},\infty}\,r^2=CKr^2,
\]
and thus, we have
\[
\int_{D_r} \abs{\nabla u}^2\le\frac{C}{r^{2+n}}\left(\int_{F_r}\abs{u}\right)^2\le\frac{CK}{r^n}\int_{F_r} \abs{u}.
\]
Then, for $r\in[4 \tau, 2]$ as above, we compute
\[
t^{\frac{n}{n-1}} \,\abs{\Omega_t\cap D_r} \le t^{\frac{n}{n-1}-2}\int_{D_r}\abs{\nabla u}^2\le CKt^{\frac{n}{n-1}-2}r^{-n}\int_{F_r} \abs{u}.
\]
Since $t=\tau^{1-n}$, the last estimate shows that
\[
t^{\frac{n}{n-1}} \,\abs{\Omega_t\cap D_r} \le CK\tau^{n-2}r^{-n}\int_{F_r} \abs{u}.
\]
We now consider an integer $j$ with
\[
2\le j \le N:= \left\lfloor 1-\frac{\ln \tau}{\ln 2}\right\rfloor,
\]
where $\lfloor\cdot\rfloor$ denotes the floor function.
Note then that $2^N \tau\in(1,2]$, and $2^j \tau \in[4 \tau,2]$ for all $j=2,\dots, N$.
Hence, if we apply the previous estimate for $r=r_j:=2^j \tau$, we obtain that
\[
t^{\frac{n}{n-1}} \,\abs{\Omega_t\cap D_{r_j}} \le CK \tau^{n-2}2^{-nj} \tau^{-n}\int_{F_{r_j}} \abs{u}=CK \tau^{-2}2^{-nj}\int_{F_{r_j}} \abs{u}.
\]
We then add those inequalities for $j=2,\dots, N$, to obtain that
\begin{align*}
t^{\frac{n}{n-1}} \Abs{\Omega_t\cap\left(B_{2^N \tau}\setminus B_{2\tau}\right)}&\le\sum_{j=2}^N t^{\frac{n}{n-1}} \,\abs{\Omega_t\cap D_{r_j}} \le CK\tau ^{-2}\sum_{j=2}^N 2^{-nj}\int_{F_{r_j}} \abs{u}\\
&=CK\tau^{-2} \int_{\Omega} \abs{u}\sum_{j=2}^N 2^{-nj}\chi_{F_{r_j}}\\
&\le CK^2\tau ^{-2}\,\Bignorm{\sum_{j=2}^N 2^{-nj}\chi_{F_{r_j}}}_{L^{\frac{n}{2},1}}\;,
\end{align*}
where we used H{\"o}lder's inequality for Lorentz norms.
Therefore, we have shown that
\begin{equation}\label{eq:OutsideB_{4t_0}}
t^{\frac{n}{n-1}} \Abs{\Omega_t\cap\left(B_{2^N \tau}\setminus B_{2\tau}\right)} \le CK^2\tau^{-2}\, \Bignorm{\sum_{j=2}^N2^{-nj}\chi_{F_{r_j}}}_{L^{\frac{n}{2},1}}.
\end{equation}
In addition,
\[
t^{\frac{n}{n-1}} \Abs{\Omega_t\cap B_{2\tau}} \le t^{\frac{n}{n-1}} \abs{B_{2\tau}} =\tau^{-n}\, (2\tau)^n\, \abs{B_1}=C 2^n.
\]
Therefore, adding with \eqref{eq:OutsideB_{4t_0}}, we obtain that
\begin{equation}\label{eq:BeforeLorentzEstimate}
t^{\frac{n}{n-1}} \Abs{\Omega_t\cap B_{2^N\tau}} \le CK^2\tau^{-2} \,\Bignorm{\sum_{j=2}^N 2^{-nj}\chi_{F_{r_j}}}_{L^{\frac{n}{2},1}}+C 2^n,
\end{equation}
where $C$ depends only on $n$, $\lambda$, $\norm{\mathbf A}_{\infty}$, $r_{\vec b-\vec c}\left(\frac{\lambda}{3C_n}\right)$, and $\abs{\Omega}$.

We now turn our attention to the last $L^{\frac{n}{2},1}$ norm. Note first that, for any $r>0$,
\[
F_r=\Omega\cap(B_{3r}(x)\setminus B_{r/4}(x))\subset \bigcup_{k=1}^6\left(\Omega\cap(B_{kr/2}\setminus B_{kr/4})\right)=\bigcup_{k=1}^6 D_{kr/2},
\]
which shows that
\[
\sum_{j=2}^N 2^{-nj}\chi_{F_{r_j}} \le \sum_{j=2}^N 2^{-nj}\sum_{k=1}^6\chi_{D_{kr_j/2}}=\sum_{k=1}^6 \sum_{j=2}^N2^{-nj}\chi_{D_{kr_j/2}}.
\]
Therefore, 
\begin{equation}\label{eq:Lorentzestimate6}
\Bignorm{\sum_{j=2}^N 2^{-nj}\chi_{F_{r_j}}}_{L^{\frac{n}{2},1}} \le C\sum_{k=1}^6\, \Bignorm{\sum_{j=2}^N2^{-nj}\chi_{D_{kr_j/2}}}_{L^{\frac{n}{2},1}}.
\end{equation}
Now, for fixed $k\in \set{1,\dots, 6}$, the sets $D_{kr_j/2}$ are pairwise disjoint, with
\[
\abs{D_{kr_j/2}}= C \left(\frac{kr_j}{2}\right)^n\le C r_j^n=C (2^j \tau)^n=C 2^{jn} \tau^n.
\]
Therefore, we obtain that, for $j=2,\dots N$,
\[
b_j:=\sum_{i=2}^j \,\abs{D_{kr_i/2}} \leq C \tau^n\sum_{i=2}^j 2^{in}\le C \tau^n 2^{jn}.
\]
As in \cite[Example 1.4.8]{Grafakos}, we then estimate
\[
\Bignorm{\sum_{j=2}^N 2^{-nj}\chi_{D_{kr_j/2}}}_{L^{\frac{n}{2},1}}\le C \sum_{j=2}^N2^{-nj} b_j^{2/n}\le C\sum_{j=2}^N2^{-nj}(\tau^n2^{nj})^{2/n}=C \tau^2\sum_{j=2}^N 2^{(2-n)j}\le C \tau^2.
\]
Hence, substituting in \eqref{eq:Lorentzestimate6}, and then substituting in \eqref{eq:BeforeLorentzEstimate}, we finally obtain that
\begin{align*}
t^{\frac{n}{n-1}} \,\abs{\Omega_t\cap B_{2^N\tau}}&\le CK^2\tau^{-2}\sum_{k=1}^6 \,\Bignorm{\sum_{j=2}^N2^{-nj}\chi_{D_{kr_j/2}}}_{L^{\frac{n}{2},1}}+ C 2^n\\
&\le CK^2\tau^{-2}\sum_{k=1}^6 \tau^2+C 2^n \le C(K^2+1).
\end{align*}
Since $2^N\tau>1$, we have thus shown that
\begin{equation}\label{eq:InsideB_1}
t^{\frac{n}{n-1}}\,\abs{\Omega_t\cap B_1}\le C(K^2+1),
\end{equation}
where $C$ depends on $n$, $\lambda$, $\norm{\mathbf A}_{\infty}$ and $r_{\vec b-\vec c}\left(\frac{\lambda}{3C_n}\right)$. On the other hand, by the second estimate in Lemma~\ref{Caccioppoli} and the second estimate in Lemma~\ref{2By1}, we have
\[
\int_{\Omega\setminus B_1}\abs{\nabla u}^2\leq C\int_{\Omega\setminus B_{2/3}}\abs{u}^2\leq C\left(\int_{\Omega\setminus B_{1/3}} \abs{u}\right)^2\leq CK^2 \abs{\Omega\setminus B_{1/3}}^{4/n}\leq CK^2,
\]
where $C$ depends on $\abs{\Omega}$ as well.
Therefore, we have (recall $t>4^{n-1}$)
\[
t^{\frac{n}{n-1}}\,\abs{\Omega_t \setminus B_1} \le t^{\frac{n}{n-1}}\, t^{-2} \int_{\Omega_t \setminus B_1} \abs{\nabla u}^2 \le CK^2.
\]
Combining together with \eqref{eq:InsideB_1}, we get $t^{\frac{n}{n-1}} \, \abs{\Omega_t} \le C(K^2+1)$ as desired.
\end{proof}

\section{Properties of variational solutions}								\label{sec4}
In this section we will construct variational $W_0^{1,2}(\Omega)$ solutions to the equations $Lu=F$ and $L\tran u=F$ for $F\in W^{-1,2}(\Omega)$. We will then show boundedness properties of the operators $T=L^{-1}$ and $S=(L\tran)^{-1}$.

\subsection{The constructions}
	
To construct variational solutions we will use the Fredholm alternative. The first step is the following.

\begin{lemma}\label{Gamma}
Let $\mathbf{A}$ be bounded and satisfy the uniform ellipticity condition \eqref{ellipticity} and $\vec b ,\vec c \in L^n(\Omega)$, $d\in L^{n/2}(\Omega)$, and $d\ge\dv\vec b$.
Let us set $\gamma=\frac{3}{4\lambda}\, r_{\vec b-\vec c}\left(\frac{\lambda}{3C_n}\right)^2$.
Then, for any $u\in W_0^{1,2}(\Omega)$, we have
\[
\int_{\Omega}\left(\mathbf{A}\nabla u \cdot \nabla u+(\vec b+\vec c) \cdot \nabla u \,u+du^2\right)+\gamma\int_{\Omega}u^2\ge\frac{\lambda}{3}\int_{\Omega} \abs{\nabla u}^2.
\]
\end{lemma}
\begin{proof}
We follow the proof of \cite[Theorem 3.2]{StampacchiaDirichlet}.
First, note that the hypothesis $d\ge\dv\vec b$ shows that
\[
\int_{\Omega}du^2+\vec b\cdot\nabla(u^2)\ge 0.
\]
Therefore,
\[
\alpha(u,u)=\int_{\Omega}\mathbf{A}\nabla u \cdot \nabla u+\vec b \cdot\nabla u\, u+\vec c\cdot\nabla u\, u+du^2\ge\int_{\Omega}\mathbf{A}\nabla u \cdot \nabla u+(\vec c-\vec b)\cdot \nabla u\, u.
\]
Splitting $\vec b-\vec c$ as in Lemma~\ref{Pointwise} and using the Cauchy inequality, we get
\begin{align*}
\int_{\Omega}\mathbf{A}\nabla u\cdot\nabla u+(\vec c-\vec b)\cdot\nabla u\, u
&\ge\lambda\int_{\Omega}\abs{\nabla u}^2-\int_{\Omega}\abs{\vec b-\vec c}_{(t)}\, \abs{\nabla u\, u}-\int_{\Omega} \abs{\vec b-\vec c}^{(t)}\, \abs{\nabla u\, u}\\
&\ge\lambda \norm{\nabla u}_2^2-t\, \norm{\nabla u}_2\,\norm{u}_2-C_n\, \norm{\,\abs{\vec b-\vec c}^{(t)}}_n \,\norm{\nabla u}_2^2\\
&\ge\left(\lambda-C_nR_{\vec c- \vec b}(t)- t\delta\right) \norm{\nabla u}_2^2-\frac{t}{4\delta} \norm{u}_2^2.
\end{align*}
By choosing $t>0$ such that $R_{\vec b-\vec c}(t)<\frac{\lambda}{3 C_n}$ and $\delta=\frac{\lambda}{3t}$, we then get 
\[
\int_{\Omega}\mathbf{A}\nabla u \cdot\nabla u+(\vec b+\vec c) \cdot \nabla u\, u+du^2 \ge
\frac{\lambda}{3}\int_{\Omega}\abs{\nabla u}^2-\frac{3 t^2}{4\lambda}\int_{\Omega}\abs{u}^2.
\]
Therefore, the inequality holds with $\gamma=3 t^2/4\lambda$ for any $t>0$ with $R_{\vec b-\vec c}(t)<\frac{\lambda}{3C_n}$.
Letting $t\to r_{\vec b-\vec c} \left(\frac{\lambda}{3C_n}\right)$ completes the proof.
\end{proof}

We now proceed to construct $W_0^{1,2}(\Omega)$ solutions.

\begin{lemma}\label{W_0^{1,2}Solvability}
Let $\Omega\subset \mathbb R^n$ be a domain with $\abs{\Omega}<+\infty$.
Let $\mathbf{A}$ be bounded and satisfy the uniform ellipticity condition \eqref{ellipticity} and $\vec b$, $\vec c\in L^n(\Omega)$, $d\in L^{n/2}(\Omega)$, and $d\ge\dv\vec b$.
Then, for every $F\in W^{-1,2}(\Omega)$, there exists a unique $u\in W_0^{1,2}(\Omega)$, such that 
\[
-\dv (\mathbf{A}\nabla u+\vec bu)+\vec c\cdot\nabla u+du=F \quad\text{in}\quad\Omega.
\]
\end{lemma}
\begin{proof}
Since we are assuming that $\abs{\Omega}<+\infty$, \cite[Lemma 11.2]{Tartar} shows that the embedding $W_0^{1,2}(\Omega)\to L^2(\Omega)$ is compact.
Then, the proof follows from the Fredholm alternative, using Lemma~\ref{Gamma}, Proposition~\ref{BoundForCriticalSubsolutions}, and Remark~\ref{rmk1344sat}.
\end{proof}

From the previous lemma, the operator
\begin{equation}\label{eq:T}
T:W^{-1,2}(\Omega)\to W_0^{1,2}(\Omega),\quad TF=u,
\end{equation}
that sends $F\in W^{-1,2}(\Omega)$ to the solution $u\in W_0^{1,2}(\Omega)$ of
\[
-\dv (\mathbf{A}\nabla u+\vec bu)+\vec c \cdot \nabla u+du=F,
\]
is well defined and injective.
$T$ is also surjective, since, if $u\in W_0^{1,2}(\Omega)$, by setting
\[
\ip{F_u, v}:=\int_{\Omega}\mathbf{A}\nabla u \cdot\nabla v+\vec b\cdot\nabla v\, u+\vec c\cdot\nabla u\, v+duv
\]
for $v\in W_0^{1,2}(\Omega)$, then $F_u\in W^{-1,2}(\Omega)$, and $TF_u=u$.

\begin{lemma}\label{W_0^{1,2}SolvabilityAdjoint}
Assume the same hypotheses as in Lemma~\ref{W_0^{1,2}Solvability}.
Then, for every $F\in W^{-1,2}(\Omega)$, there exists a unique $u\in W_0^{1,2}(\Omega)$, such that
\[
-\dv(\mathbf{A}\tran\nabla u+\vec cu)+\vec b \cdot \nabla u+du=F \quad \text{in}\quad \Omega.
\] 
\end{lemma}
\begin{proof}
Let $\alpha$ be the bilinear form that corresponds to $L$ as in the proof of Lemma~\ref{Gamma}, and $\alpha\tran$ be the bilinear form that corresponds to $L\tran$.
Then we have
\[
\alpha\tran(v,u)=\alpha(u,v)\quad \forall u, v \in W^{1,2}_0(\Omega).
\]
Suppose that $\alpha\tran(u,v)=0$ for some $u\in W_0^{1,2}(\Omega)$ and for all $v\in W_0^{1,2}(\Omega)$. Then, for all $F\in W^{-1,2}(\Omega)$,
\[
0=\alpha\tran(u,TF)=\alpha(TF,u)=\ip{F,u},
\]
which implies that $u=0$.
Therefore, we see that $u=0$ is the unique solution in $W^{1,2}_0(\Omega)$ for the equation $L\tran u=0$.
Then, the lemma follows from the Fredholm alternative.
\end{proof}

Let $S$ be the operator
\begin{equation}\label{eq:S}
S:W^{-1,2}(\Omega)\to W_0^{1,2}(\Omega),\quad SF=u,
\end{equation}
that sends $F\in W^{-1,2}(\Omega)$ to the solution $u\in W_0^{1,2}(\Omega)$ of
\[
-\dv(\mathbf{A}\tran\nabla u+\vec cu)+\vec b\cdot\nabla u+du=F.
\]
Then, an integration by parts argument shows that the $W^{-1,2}(\Omega)\to W_0^{1,2}(\Omega)$ adjoint of $T$ is equal to $S$. 
Moreover, considering the embedding $i:L^2(\Omega)\to W^{-1,2}(\Omega)$, we obtain that
\begin{equation}\label{eq:AdjointRelation}
\int_{\Omega}Tf_1\,f_2=\int_{\Omega}f_1\,Sf_2,\quad \forall f_1,f_2\in L^2(\Omega).
\end{equation}

Finally, we show the next lemma, which gives an a priori bound on the gradients of solutions to $Lu=f$.

\begin{lemma}\label{BoundOnU}
	Let $\mathbf{A}$ be bounded and satisfy the uniform ellipticity condition \eqref{ellipticity} and $\vec b$, $\vec c\in L^n(\Omega)$, $d\in L^{n/2}(\Omega)$, and  $d\ge\dv\vec b$.
	For $f\in L^2(\Omega)$, let $u\in W_0^{1,2}(\Omega)$ be a solution to the equation
	\[
	-\dv(\mathbf{A}\nabla u+\vec bu)+\vec c \cdot\nabla u+du=f\quad \text{in}\quad \Omega.
	\]
	Then,
	\[
	\norm{\nabla u}_{L^2(\Omega)}\le C\left(\norm{f}_{L^2(\Omega)}+\norm{u}_{L^2(\Omega)}\right),
	\]
	where $C$ depends only on $n$, $\lambda$, and $r_{\vec b-\vec c}\left(\frac{\lambda}{3 C_n}\right)$.
\end{lemma}
\begin{proof}
	Note that, for the $\gamma$ that appears in Lemma \ref{Gamma}, we have
	\[
	\frac{\lambda}{3}\int_{\Omega} \abs{\nabla u}^2\le\int_{\Omega}\left(\mathbf{A}\nabla u \cdot\nabla u+(\vec b+\vec c)\cdot \nabla u\, u+du^2\right)+\gamma\int_{\Omega}u^2.
	\]
	Hence, if $u$ solves the equation $Lu=f$, we obtain that
	\[
	\frac{\lambda}{3}\int_{\Omega} \abs{\nabla u}^2\le\int_{\Omega}fu+\gamma\int_{\Omega}u^2,
	\]
	and the claim follows.
\end{proof}

\subsection{Boundedness of solutions}
We now turn to boundedness of $T$, when we restrict its domain to $W^{-1,q'}(\Omega)$, where $q'$ is the conjugate exponent to $q$.
Note first that, from \cite[Theorem 3.9]{Adams}, every $F\in W^{-1,q'}$ can be written as $f+\dv \vec g$ with $f$, $\vec g \in L^{q'}$ and $\norm{F}_{W^{-1,q'}}$ is controlled by $\norm{f}_{q'} + \norm{\vec g}_{q'}$.
The following corollary follows from Proposition~\ref{BoundForCriticalSubsolutions} and Remark~\ref{rmk1344sat}.
\begin{corollary}	\label{BoundForCritical}
Let $\Omega\subset \mathbb R^n$ be a domain with $\abs{\Omega}<+\infty$.
Let $\mathbf{A}$ be bounded and satisfy the uniform ellipticity condition \eqref{ellipticity} and $\vec b$, $\vec c\in L^n(\Omega)$, $d\in L^{n/2}(\Omega)$, and $d\ge\dv\vec b$.
Then, for $1<q< \frac{n}{n-1}$, the operator $T$ defined in \eqref{eq:T} maps $W^{-1,q'}(\Omega)$ to $L^{\infty}(\Omega)$, with its operator norm bounded by a constant depending only on $n$, $q$, $\lambda$, $r_{\vec b-\vec c}\left(\frac{\lambda}{3 C_n}\right)$, and $\abs{\Omega}$.
\end{corollary}

By considering the dual operator to $T$, we obtain the next corollary.

\begin{corollary}\label{MuBound}
Assume the same hypotheses as in Corollary~\ref{BoundForCritical}.
Then, for $1<q< \frac{n}{n-1}$, the operator $T^*$ maps $L^1(\Omega)$ to $W_0^{1,q}(\Omega)$ with its operator norm bounded by a constant depending only on $n$, $q$, $\lambda$, and $r_{\vec b-\vec c}\left(\frac{\lambda}{3 C_n}\right)$, and $\abs{\Omega}$.

\end{corollary}
\begin{proof}
Note first that $W_0^{1,q}(\Omega)$ is reflexive; see e.g. \cite[Theorem 3.6]{Adams}.
By Corollary~\ref{BoundForCritical}, we then see that $T^*$ maps $L^1(\Omega)\subsetneq \left(L^{\infty}(\Omega)\right)^*$ to $W_0^{1,q}(\Omega)$ with $\norm{T^*}=\norm{T}$.
\end{proof}

Similar corresponding results hold for the adjoint equation if we assume that $\vec c, d$ are regular enough.

We remark that Lemma~\ref{BeforeMuBoundAdjoint} below will only be used qualitatively.
For this reason, we do not indicate in the proof how the operator norm $\norm{S}_{W^{-1,q'}(\Omega)\to L^{\infty}(\Omega)}$ depends on the coefficients.
Also, we have not attempted to find the optimal conditions for the coefficients $\vec c$ and $d$ for the same reason; they can be relaxed but it is not important for our purposes.

\begin{lemma}\label{BeforeMuBoundAdjoint}
Let $\Omega\subset\mathbb R^n$ be a domain with $\abs{\Omega}<+\infty$.
Let $\mathbf{A}$ be bounded and satisfy the uniform ellipticity condition \eqref{ellipticity} and $\vec b\in L^n(\Omega)$, $\vec c ,d\in\Lip(\Omega)\cap L^{\infty}(\Omega)$, and $d\geq\dv \vec b$.
Then the operator $S$, defined in \eqref{eq:S}, is bounded from $W^{-1,q'}(\Omega)$ to $L^{\infty}(\Omega)$ for $1<q< \frac{n}{n-1}$.
\end{lemma}

\begin{proof}
Using a procedure analogous to the proof of Lemma \ref{Gamma}, we can find $\gamma>0$ such that the bilinear form
\[
\alpha_{\gamma}\tran(u,v)=\int_{\Omega}\mathbf{A}\tran\nabla u \cdot \nabla v+\vec c\cdot\nabla v\, u+\vec b \cdot\nabla u\, v+(d+\gamma)uv
\]
is coercive in $W_0^{1,2}(\Omega)\times W_0^{1,2}(\Omega)$.
Also, since $\vec c,d\in\Lip(\Omega)\cap L^\infty(\Omega)$, we can find $\gamma$ such that
\[
d+\gamma\ge\dv\vec c.
\]
Then, Proposition~\ref{BoundForCriticalSubsolutions} shows that, for any $q\in (1,\frac{n}{n-1})$, there exists a constant $C$ such that, if $F\in W^{-1,q'}(\Omega)$ and $u\in W_0^{1,2}(\Omega)$ is a solution to the equation
\[
-\dv(\mathbf{A}\tran\nabla u+\vec cu)+\vec b\cdot\nabla u+(d+\gamma)u=F,
\]
then
\[
\sup_{\Omega}\, \abs{u}\le C \norm{F}_{W^{-1,q'}(\Omega)}.
\]
Set $S_{\gamma}:W^{-1,2}(\Omega)\to W_0^{1,2}(\Omega)$ to be the operator that sends $F$ to $u$ above.
Then $S_\gamma$ is bounded from $W^{-1,q'}(\Omega)$ to $L^{\infty}(\Omega)$. 
Consider also the inclusion operator
\[
i:L^{\infty}(\Omega)\to W^{-1,q'}(\Omega),\quad \ip{i(u),v}=\int_{\Omega}uv.
\]
Since the embedding $W_0^{1,q}(\Omega)\hookrightarrow L^q(\Omega)$ is compact (see e.g., \cite[Lemma 11.2]{Tartar}), the embedding $L^{q'}(\Omega)\hookrightarrow W^{-1,q'}(\Omega)$ is also compact.
Therefore the operator $i$ is compact and thus, if we set
\[
\tilde{S}_{\gamma}=i\circ S_{\gamma}:W^{-1,q'}(\Omega)\to W^{-1,q'}(\Omega),
\]
then $\tilde{S}_{\gamma}$ is compact.
	
Now, we show that $\ker\left(I-\gamma\tilde{S}_{\gamma}\right)=\set{0}$.
Suppose $F\in W^{-1,q'}(\Omega)$ satisfies $F=\gamma\tilde{S}_{\gamma}F$.
Then, $i(\gamma S_{\gamma}F)=F$, and thus, we have
\[
\gamma\int_{\Omega}S_{\gamma}F\, v=\ip{F,v},\quad \forall v\in W_0^{1,q}(\Omega).
\]
Since  $W_0^{1,2}(\Omega)\subset W_0^{1,q}(\Omega)$ and $W^{-1,q'}(\Omega)\subset W^{-1,2}(\Omega)$, we find that $F\in W^{-1,2}(\Omega)$ and 
\[
\gamma\int_{\Omega}S_{\gamma}F\, v=\ip{F,v},\quad \forall v\in W_0^{1,2}(\Omega).
\]
Therefore, for any $\phi\in C_c^{\infty}(\Omega)$, we have
\[
\alpha\tran(S_{\gamma}F,\phi)=\alpha_{\gamma}\tran (S_{\gamma}F,\phi)-\gamma\int_{\Omega}S_{\gamma}F\,\phi=\ip{F,\phi}-\gamma\int_{\Omega}S_{\gamma}F\, \phi=0.
\]
Note that our hypotheses imply that Lemma \ref{W_0^{1,2}SolvabilityAdjoint} is applicable, therefore we obtain that $S_\gamma F=0$ and thus $F=0$, which shows that $I-\gamma\tilde{S}_{\gamma}$ is injective as claimed.
Then, the Fredholm alternative shows that the operator
\[
I-\gamma\tilde{S}_{\gamma}:W^{-1,q'}(\Omega)\to W^{-1,q'}(\Omega)
\]
is surjective.
Hence, the open mapping theorem implies that there exists $C>0$ such that, for every $F\in W^{-1,q'}(\Omega)$, there exists $G\in W^{-1,q'}(\Omega)$, with
\[
G-\gamma\tilde{S}_{\gamma}G=F,\quad \norm{G}_{W^{-1,q}(\Omega)}\le C\norm{F}_{W^{-1,q}(\Omega)}.
\]
If we set $u=S_{\gamma}G$, then since $S_{\gamma}:W^{-1,q'}(\Omega)\to L^{\infty}(\Omega)$ is bounded, we have
\[
\norm{u}_{L^{\infty}(\Omega)}\le C \norm{G}_{W^{-1,q'}(\Omega)}\le C \norm{F}_{W^{-1,q'}(\Omega)}.
\]
Moreover, $u\in W_0^{1,2}(\Omega)$ and for every $\phi\in C_c^{\infty}(\Omega)$, we have
\[
\alpha\tran(u,\phi)=\alpha_{\gamma}\tran(S_{\gamma}G,\phi)-\int_{\Omega}\gamma S_{\gamma}G\,\phi= \ip{G, \phi}-\gamma \ip{i\circ S_{\gamma}(G), \phi}= \ip{G-\gamma\tilde{S}_{\gamma}G, \phi}=\ip{F,\phi}.
\]
Hence, $u$ solves the equation $L\tran u=F$.
Since from Lemma~\ref{W_0^{1,2}SolvabilityAdjoint} solutions for $F\in W^{-1,2}(\Omega)$ are unique, we obtain that $u=SF$.
Therefore, $S:W^{-1,q'}(\Omega)\to L^{\infty}(\Omega)$ is bounded.
\end{proof}

By considering the dual operator to $S$, we obtain the next corollary.

\begin{corollary}\label{MuBoundAdjoint}
Let $\Omega\subset\mathbb R^n$ be a domain with $\abs{\Omega} <+\infty$.
Let $\mathbf{A}$ be bounded and satisfy the uniform ellipticity condition \eqref{ellipticity} and $\vec b\in L^n(\Omega)$, $\vec c,d\in\Lip(\Omega)\cap L^{\infty}(\Omega)$, and $d\ge\dv\vec b$.
Then, the operator $S^*$ maps $L^1(\Omega)$ to $W_0^{1,q}(\Omega)$ for  $1<q<\frac{n}{n-1}$.
\end{corollary}
\begin{proof}
The proof follows from Lemma \ref{BeforeMuBoundAdjoint}, as in the proof of Corollary \ref{MuBound}.
\end{proof}

\section{The main constructions}										\label{sec5}
\subsection{Definition and construction of Green's function} We now turn to the definition of Green's function. 

\begin{definition}\label{GreenDefinition}
Let $\Omega\subset \mathbb R^n$ be a domain with $\abs{\Omega}<+\infty$.
Let $\mathbf{A}$ be bounded and satisfy the uniform ellipticity condition \eqref{ellipticity} and $\vec b$, $\vec c\in L^n(\Omega)$, $d\in L^{n/2}(\Omega)$, and $d\ge\dv\vec b$.
We call a function $G_y(x)=G(x,y)$ to be Green's function for the equation $Lu=0$ in $\Omega$, if $G_y\in L^1(\Omega)$ for almost every $y\in\Omega$, and also
\[
S\phi(y)=\int_{\Omega}G(x,y)\phi(x)\,dx,
\]
for every $\phi\in L^{\infty}(\Omega)$, where $S$ is defined in \eqref{eq:S}.
Similarly, we call a function $g_x(y)=g(y,x)$ to be Green's function for the equation $L{\tran}u=0$ in $\Omega$, if $g_x\in L^1(\Omega)$ for almost every $x\in\Omega$, and also
\[
T\phi(x)=\int_{\Omega}g(y,x)\phi(y)\,dy,
\]
for every $\phi\in L^{\infty}(\Omega)$, where $T$ is defined in \eqref{eq:T}.
\end{definition}

We have not yet established existence of Green's functions.
However, if $G$ and $\tilde G$ are Green's functions for $L$ in $\Omega$, then for almost every $y\in\Omega$ and every $\phi\in C_c^{\infty}(\Omega)$,
\[
\int_{\Omega}\left(G(x,y)-\tilde{G}(x,y)\right)\phi(x)\,dx=0\Rightarrow G_y=\tilde{G}_y.
\]
Therefore, Green's function, if it exists, is unique up to sets of measure zero.
The same holds for Green's function for $L{\tran}$.

We now apply Corollary~\ref{MuBound} to the function
\[
f_k=\abs{B_{1/k}(x)}^{-1}\chi_{B_{1/k}(x)}
\]
to obtain a solution $g_x^k=Sf_k\in W_0^{1,2}(\Omega)$ to the equation
\[
-\dv(\mathbf{A}\tran \nabla g_x^k+\vec cg_x^k)+\vec b\cdot\nabla g_x^k+dg_x^k=f_k\quad\text{in}\quad\Omega
\]
with
\begin{equation}\label{eq:g_x^kBound}
\norm{g_x^k}_{W_0^{1,q}(\Omega)}\le C \norm{f_k}_{L^1(\Omega)}\le C,
\end{equation}
where $C$ depends only on $n$, $q$, $\lambda$, $r_{\vec b-\vec c}\left(\frac{\lambda}{3C_n}\right)$, and $\abs{\Omega}$.

We assume that $k$ is large enough that $B_{1/k}(x)\subset \Omega$.
In addition, since $\set{g_x^k}$ is bounded in $W_0^{1,q}(\Omega)$, it has a subsequence $\set{g_x^{i_k}}$ that converges weakly in $W_0^{1,q}(\Omega)$, strongly in $L^q(\Omega)$, and almost everywhere in $\Omega$, to a function $g_x\in W_0^{1,q}(\Omega)$. Then, for every $\phi\in L^{\infty}(\Omega)$, we have
\begin{equation*}
\int_{\Omega}g_x\phi=\lim_{k\to\infty}\int_{\Omega}g_x^{i_k}\phi =\lim_{k\to\infty}\int_{\Omega}Sf_{i_k}\,\phi =\lim_{k\to\infty}\int_{\Omega}T\phi\, f_{i_k}=\lim_{k\to\infty}\fint_{B_{1/i_k}(x)}T\phi=T\phi(x),
\end{equation*}
for almost every $x\in\Omega$, where we used \eqref{eq:AdjointRelation} in the third equality, and Lebesgue's differentiation theorem.
Setting $g(y,x)=g_x(y)$, we are led to the following proposition.

\begin{proposition}\label{g_x}
Let $\Omega\subset \mathbb R^n$ be a domain with $\abs{\Omega} <+\infty$.
Let $\mathbf{A}$ be bounded and satisfy the uniform ellipticity condition \eqref{ellipticity} and $\vec b$, $\vec c\in L^n(\Omega)$, $d\in L^{n/2}(\Omega)$, and $d\ge \dv \vec b$.
Then there exists a function $g(\cdot,\cdot)$ on $\Omega\times \Omega$
such that the solution $u\in W_0^{1,2}(\Omega)$ of the equation $Lu=\phi$, for $\phi\in L^{\infty}(\Omega)$, is given by
\[
u(x)=\int_{\Omega}g(y,x)\phi(y)\,dy
\]
for almost every $x\in\Omega$. 
Moreover, for $1\le q <\frac{n}{n-1}$, there exists a constant $C_q$ depending only on $n$, $q$, $\lambda$, $r_{\vec b-\vec c}\left(\frac{\lambda}{3C_n}\right)$, and $\abs{\Omega}$ such that
\[
\norm{g(\cdot,x)}_{W_0^{1,q}(\Omega)}\le C_q
\]
uniformly in $x$.
\end{proposition}

We shall call $g(\cdot,\cdot)$ to be \textit{Green's function} for the operator $L\tran$ in $\Omega$.

In the following, we will need to apply the maximum principle in some cases in order to obtain the pointwise bounds for Green's function. Since solutions to the equation $L\tran u=0$ do not necessarily satisfy the maximum principle, we will construct Green's function for the operator $L$, first under an additional qualitative assumption that the lower order coefficients are regular enough. In order to do this, we apply Corollary~\ref{MuBoundAdjoint}, to obtain that, if $\vec c,d\in\Lip(\Omega)\cap L^{\infty}(\Omega)$, and $1<q<\frac{n}{n-1}$, then there exists a solution $G_y^m=Th_m\in W_0^{1,2}(\Omega)$ to the equation
\[
-\dv(\mathbf{A}\nabla u+\vec bu)+\vec c \cdot \nabla u+du=h_m:=\abs{B_{1/m}(y)}^{-1}\chi_{B_{1/m}(y)},
\]
with
\begin{equation}\label{eq:G_y^mBound}
\norm{G_y^m}_{W_0^{1,q}(\Omega)}\le C\norm{h_m}_{L^1(\Omega)}=C.
\end{equation}
Since $\set{G_y^m}$ is bounded in $W_0^{1,q}(\Omega)$, it has a subsequence $\set{G_y^{j_m}}$ that converges weakly in $W_0^{1,q}(\Omega)$, strongly in $L^q(\Omega)$, and almost everywhere in $\Omega$, to a function $G_y\in W_0^{1,q}(\Omega)$.
Then, we are led to the following proposition.

\begin{proposition}\label{G_y}
Let $\Omega\subset \mathbb R^n$ be a domain with $\abs{\Omega} <+\infty$.
Let $\mathbf{A}$ be bounded and satisfy the uniform ellipticity condition \eqref{ellipticity} and $\vec b\in L^n(\Omega)$, $\vec c,d\in\Lip(\Omega)\cap L^{\infty}(\Omega)$, and $d\ge\dv\vec b$.
Then there exists a function $G(\cdot,\cdot)$ on $\Omega\times \Omega$
such that the solution $u\in W_0^{1,2}(\Omega)$ of the equation $L\tran u=\phi$, for $\phi\in L^{\infty}(\Omega)$, is given by
\[
u(y)=\int_{\Omega}G(x,y)\phi(x)\,dx
\]
for almost every $y\in\Omega$.
Moreover, for $1 \le q< \frac{n}{n-1}$, the norms $\norm{G(\cdot,y)}_{W_0^{1,q}(\Omega)}$ are bounded uniformly in $y$.
\end{proposition}

\begin{proof}
Set $G(x,y)=G_y(x)$, the function constructed above.
Then, the $W_0^{1,q}(\Omega)$ bound follows from \eqref{eq:G_y^mBound}.
Moreover, for any $\phi\in L^{\infty}(\Omega)$, we have
\[
\int_{\Omega}G_y\phi=\lim_{m\to\infty}\int_{\Omega}G_y^{j_m}\phi=\lim_{m\to\infty}\int_{\Omega}Th_{j_m}\,\phi=\lim_{m\to\infty}\fint_{B_{1/j_m}(y)} S\phi=S\phi(y),
\]
for almost every $y\in\Omega$ by \eqref{eq:AdjointRelation} and Lebesgue's differentiation theorem.
\end{proof}

At this moment, we will divert from the critical case in order to obtain the next lemma.

\begin{lemma}\label{Presymmetry}
Let $\Omega\subset \mathbb R^n$ be a domain with $\abs{\Omega} <+\infty$.
Let $\mathbf{A}$ be bounded and satisfy the uniform ellipticity condition \eqref{ellipticity} and $\vec b$, $\vec c$, $d\in \Lip(\Omega)\cap L^{\infty}(\Omega)$, and 
$d \ge \dv \vec b$.
Then for every $x,y\in\Omega$, with $x\neq y$, and $m$ large enough, we have
\[
G_y^m(x)=\lim_{k\to\infty}\fint_{B_{1/m}(y)}g_x^k=\fint_{B_{1/m}(y)}g_x.
\]
Moreover, for every $x,y\in\Omega$, with $x\neq y$, and $k$ large enough, we have
\[
g_x^k(y)=\lim_{m\to\infty}\fint_{B_{1/k}(x)}G_y^m=\fint_{B_{1/k}(x)}G_y.
\]
In particular, we have $g_x^k\ge 0$.
\end{lemma}
\begin{proof}
Let $x$, $y\in\Omega$ and consider the functions $g_x^k$ and $G_y^m$ in \eqref{eq:g_x^kBound} and \eqref{eq:G_y^mBound}.
Recall that $g_x^k$ and $G_y^m$ belong to $W_0^{1,2}(\Omega)$, and they solve the equations
\begin{align*}
\int_{\Omega}\mathbf{A}\tran\nabla g_x^k \cdot\nabla\phi+\vec c\cdot\nabla\phi\, g_x^k+\vec b\cdot\nabla g_x^k\,\phi+dg_x^k\phi &=\fint_{B_{1/k}(x)}\phi,\\
\int_{\Omega}\mathbf{A}\nabla G_y^m\cdot \nabla\psi+\vec b\cdot\nabla\psi\, G_y^m+\vec c\cdot\nabla G_y^m\,\psi+dG_y^m\psi &=\fint_{B_{1/m}(y)}\psi,
\end{align*}
for $\phi$, $\psi\in W_0^{1,2}(\Omega)$.
Setting $\phi=G_y^m$ and $\psi=g_x^k$, we see that the left hand sides in the above coincide, and thus we have
\begin{equation}\label{eq:BasSym}
\fint_{B_{1/k}(x)}G_y^m=\fint_{B_{1/m}(y)}g_x^k.
\end{equation}
Consider now $x\neq y$ in $\Omega$, and fix $m\in\mathbb N$ such that $\frac{1}{m}<\frac{1}{2} \abs{x-y}$.
Then, since $LG_y^m=0$ in $\Omega\setminus B_{1/m}(y)$, Theorem 7.1 in \cite{StampacchiaDirichlet} shows that $G_y^m$ is continuous at $x$.
Hence, letting $k\to\infty$ in \eqref{eq:BasSym}, we obtain that
\[
G_y^m(x)=\lim_{k\to\infty}\fint_{B_{1/k}(x)}G_y^m=\lim_{k\to\infty}\fint_{B_{1/m}(y)}g_x^k.
\]
This shows the first equality. 
Moreover, since a subsequence of $g_x^k$ converges to $g_x$ in $L^1(\Omega)$, we obtain the second equality. To show the third and fourth equalities, we follow a similar procedure. 
	
Finally, since $G_y^m$ satisfies $LG_y^m\ge 0$ in $\Omega$, and $G_y^m$ vanishes on the boundary, Proposition~\ref{BoundForCriticalSubsolutions} applied to $-G_y^m$ shows that $G_y^m\ge 0$ in $\Omega$, hence $g_x^k\ge 0$ in $\Omega$.
\end{proof}

\subsection{A Lorentz bound}
The aim of this subsection is to obtain a good estimate on a Lorentz norm of Green's function when the coefficients are nice enough. We begin with the following lemma, which we will need later on.

\begin{lemma}\label{f(s)Bound}
Suppose that $a>0$, $f:[a,\infty)\to\mathbb R_+$ is non-increasing, and
\[
f(2t)\le\frac{c_1}{t}+\frac{1}{3}f(t),\quad \forall t \ge a,
\]
for some $c_1>0$. 
Then $f(t)\le\dfrac{C}{t}\,$ for all $t\ge a$, where
\[
C=6c_1+2af(a).
\]
\end{lemma}
\begin{proof}
By induction on $n$, we show that
\begin{equation}\label{eq:Ind}
f(2^{n+1}a)\le\frac{c_1}{2^na}\sum_{k=0}^n\frac{2^k}{3^k}+\frac{1}{3^n}f(a).
\end{equation}
Let now $t\ge a$, then there exists $m\in\mathbb N$ such that $2^ma\le t<2^{m+1}a$. Then, since $f$ is nonincreasing, and also using \eqref{eq:Ind}, we obtain that
\begin{align*}
f(t)&\le f(2^{m+1}a)\le\frac{c_1}{2^ma}\sum_{k=0}^m\frac{2^k}{3^k}+\frac{1}{3^m}f(a)\le\frac{c_1}{2^ma}\sum_{k=0}^{\infty}\frac{2^k}{3^k}+\frac{1}{3^m}f(a)\\
&=\frac{3c_1}{2^ma}+\frac{1}{3^m}f(a)=\frac{6c_1}{2^{m+1}a}+\frac{2^{m+1}}{3^m}\frac{1}{2^{m+1}a}af(a)\le\frac{6c_1+2af(a)}{t},
\end{align*}
which completes the proof.
\end{proof}

\begin{proposition}\label{GreenConstructionLipschitz}
Let $\Omega\subset \mathbb R^n$ be a domain with $\abs{\Omega} <+\infty$.
Let $\mathbf{A}$ be bounded and satisfy the uniform ellipticity condition \eqref{ellipticity} and $\vec b$, $\vec c$, $d\in \Lip(\Omega)\cap L^{\infty}(\Omega)$.
Assume also that $d\ge\dv\vec b$.
Then, for almost every $x\in\Omega$, the functions $g_x^k$ in \eqref{eq:g_x^kBound} satisfy the uniform bound
\[
\norm{g_x^k}_{L^{\frac{n}{n-2},\infty}(\Omega)}\le C,
\]
where $C$ depends only on $n$, $\lambda$, $\norm{\mathbf A}_{\infty}$, $r_{\vec b-\vec c}\left(\frac{\lambda}{3C_n}\right)$, $\tilde{r}_{\vec b-\vec c}\left(\frac{\lambda}{3C_n}\right)$, and $\abs{\Omega}$.
\end{proposition}
\begin{proof}
We adapt an argument used in Gr\"uter and Widman \cite{Gruter}.
By Lemma~\ref{Presymmetry}, we have $g_x^k\ge 0$ and thus $g_x\ge 0$.
For $s>0$, let
\[
\phi=\left(\frac{1}{s}-\frac{1}{g_x^k}\right)^+\quad \text{and}\quad 
\Omega_s^k=\set{y\in \Omega: g_x^k(y) \ge s}.
\]
Then, using $\phi$ as a test function, we obtain
\[
\int_{\Omega_s^k}\mathbf{A}\tran\nabla g_x^k \cdot\nabla\phi+\vec c\cdot \nabla\phi\, g_x^k+\vec b\cdot\nabla g_x^k\,\phi+dg_x^k\phi=\fint_{B_{1/k}(x)}\phi. 
\]
Therefore, we have
\begin{align*}
\int_{\Omega_s^k}\mathbf{A}\tran\nabla g_x^k\cdot\nabla\phi&=\fint_{B_{1/k}(x)}\phi-\int_{\Omega_s^k}\vec c \cdot\nabla\phi\, g_x^k+\vec b \cdot \nabla g_x^k\,\phi+d g_x^k\phi\\
&=\fint_{B_{1/k}(x)}\phi-\int_{\Omega_s^k}\vec b\cdot \nabla(g_x^k\phi)+d g_x^k\phi+\int_{\Omega_s^k}(\vec b-\vec c)\cdot\nabla \phi\, g_x^k.
\end{align*}
Then, using $0\le\phi\le \frac{1}{s}$ and $g_x^k\ge 0$ together with the assumption $d\ge\dv\vec b$, we get
\begin{equation*}
\int_{\Omega_s^k}\mathbf{A}\tran\nabla g_x^k \cdot\nabla\phi\le\frac{1}{s}+\int_{\Omega_s^k}(\vec b-\vec c) \cdot \nabla\phi\, g_x^k.
\end{equation*}
Since $\nabla\phi=\nabla g_x^k/(g_x^k)^2$ in $\Omega^k_s$, we thus obtain
\[
\int_{\Omega_s^k}\mathbf{A}\tran\nabla g_x^k \cdot \nabla g_x^k\, (g_x^k)^{-2} \le\frac{1}{s}+\int_{\Omega_s^k} \abs{\vec b-\vec c}\,\abs{\nabla g_x^k}\,(g_x^k)^{-1}.
\]
Set $w=\left(\ln\, (g_x^k/s)\right)^+$ so that $\nabla w=\nabla g_x^k/g_x^k$.
Then, we get
\[
\int_{\Omega_s^k}\mathbf{A}\tran\nabla w \cdot \nabla w \le\frac{1}{s}+\int_{\Omega_s^k} \abs{\vec b-\vec c}\,\abs{\nabla w}\le\frac{1}{s}+\frac{1}{2\lambda}\int_{\Omega_s^k} \abs{\vec b-\vec c}^2+\frac{\lambda}{2}\int_{\Omega_s^k} \abs{\nabla w}^2,
\]
and thus, using ellipticity of $\mathbf{A}$, we obtain
\begin{equation*}
\frac{\lambda}{2}\int_{\Omega_s^k} \abs{\nabla w}^2\le\frac{1}{s}+\frac{1}{2\lambda}\int_{\Omega_s^k} \abs{\vec b-\vec c}^2.
\end{equation*}
Since $\Omega^k_{2s} \subset \Omega^k_s$, by the Sobolev inequality and H{\"o}lder's inequality, we then get
\[
\frac{\lambda}{2} (\ln^2 2)\, \abs{\Omega_{2s}^k}^{\frac{n-2}{n}}\le\frac{\lambda}{2}\left(\int_{\Omega_{2s}^k} \abs{w}^{\frac{2n}{n-2}}\right)^{\frac{n-2}{n}}\le\frac{C_n^2}{s}+\frac{C_n^2}{2\lambda}\left(\int_{\Omega_s^k} \abs{\vec b-\vec c}^n\right)^{2/n} \abs{\Omega_s^k}^{\frac{n-2}{n}}.
\]
Therefore, if we set $f(s)=\abs{\Omega_s^k}^{\frac{n-2}{n}}$, then
\begin{equation}\label{eq:f(2s)}
f(2s)\le\frac{2C_n^2}{\lambda\ln^2 2}\,\frac{1}{s}+\frac{C_n^2}{\lambda^2\ln^2 2}\left(\int_{\Omega_s^k}\abs{\vec b-\vec c}^n\right)^{2/n}f(s).
\end{equation}
We now use Chebyshev's inequality, to obtain that
\begin{equation}\label{eq:dfnOfC}
\abs{\Omega_s^k}=\int_{\Omega_s^k} 1\,dx\le\frac{1}{s}\int_{\Omega_s^k}g_x^k\le\frac{1}{s} \norm{g_x^k}_{L^1(\Omega)}\le\frac{C}{s},
\end{equation}
where $C$ is a constant obtained from \eqref{eq:g_x^kBound} by fixing $q \in (1, \frac{n}{n-1})$; hence it depends only on $n$, $\lambda$, $r_{\vec b-\vec c}\left(\frac{\lambda}{3C_n}\right)$, and $\abs{\Omega}$.
Now, let
\[
\textstyle t_0=\tilde{r}_{\vec b-\vec c} \left(\frac{\lambda}{3C_n}\right)\quad \text{and}\quad  s_0= 2C/t_0,
\]
where $\tilde{r}_{\vec b-\vec c}$ is as defined in \eqref{eq:R_f2} and $C$ is the same constant as in \eqref{eq:dfnOfC}.
Then, we have $C/s_0 =t_0/2 <t_0$ and thus, from \eqref{eq:dfnOfC}, we see that $\abs{\Omega_s^k}<t_0$ whenever $s \ge s_0$.
Note also that $3 \ln^2 2 >1$.
Therefore, for all $s \ge s_0$, we have
\[
\frac{C_n^2}{\lambda^2\ln^2 2}\left(\int_{\Omega_s^k} \abs{\vec b-\vec c}^n\right)^{2/n}\le\frac{C_n^2}{\lambda^2\ln^2 2}\left(\tilde{R}_{\vec b-\vec c}(t_0)\right)^2\le\frac{C_n^2}{\lambda^2\ln^2 2}\left(\frac{\lambda}{3C_n}\right)^2<\frac{1}{3}.
\]
Therefore, plugging in \eqref{eq:f(2s)}, we obtain that
\[
f(2s)\le\frac{c_1}{s}+\frac{1}{3}f(s),\quad \forall s\ge s_0,
\]
where $c_1=2C_n^2/\lambda \ln^2 2$.
Therefore, from Lemma \ref{f(s)Bound}, we obtain that
\[
f(s)\le\frac{6c_1+2s_0f(s_0)}{s},\quad \forall s \ge s_0.
\]
Since $f(s_0)=\abs{\Omega_{s_0}^k}^{\frac{n-2}{n}}\le\abs{\Omega}^{\frac{n-2}{n}}$, we finally obtain that
\[
s\abs{\Omega_s^k}^{\frac{n-2}{n}}\le 6c_1+2s_0f(s_0)\le 6c_1+2s_0 \abs{\Omega}^{\frac{n-2}{2}}, \quad \forall s \ge s_0.
\]
On the other hand, if $s\le s_0$, then
\[
s\abs{\Omega_s^k}^{\frac{n-2}{n}}\le s_0 \abs{\Omega}^{\frac{n-2}{2}} \le 6c_1+2s_0 \abs{\Omega}^{\frac{n-2}{2}},
\]
which shows that the previous estimate holds for all $s\ge 0$.
Therefore, we have
\[
\norm{g_x^k}_{L^{\frac{n}{n-2},\infty}(\Omega)}\le C
\]
uniformly in $k$, where $C$ depends only on $n$, $\lambda$, $r_{\vec b-\vec c}\left(\frac{\lambda}{3C_n}\right)$, $\tilde{r}_{\vec b-\vec c}\left(\frac{\lambda}{3C_n}\right)$, and $\abs{\Omega}$.
Since $g_x^{i_k}\to g_x$ almost everywhere in $\Omega$, the last estimate also reveals that
\[
\norm{g_x}_{L^{\frac{n}{n-2},\infty}(\Omega)}\le C.\qedhere
\]
\end{proof}

\section{The subcritical case}			\label{sec6}

In this section we will treat the subcritical case, in which we will assume that $\vec b-\vec c\in L^p$ for some $p>n$.

\subsection{A preliminary bound}
We first show a pointwise bound on $g$ in the case that $\vec b,\vec c,d$ are Lipschitz functions.

\begin{proposition}\label{GreenConstructionLipschitzSubcritical}
Let $\Omega\subset \mathbb R^n$ be a domain with $\abs{\Omega} <+\infty$.
Let $\mathbf{A}$ be bounded and satisfy the uniform ellipticity condition \eqref{ellipticity} and $\vec b$, $\vec c$, $d\in \Lip(\Omega)\cap L^{\infty}(\Omega)$, with $d \ge \dv \vec b$.
Let $p>n$.
Then, for almost every $x\in\Omega$, the function $g_x$ in Proposition~\ref{g_x} satisfies the bound
\[
0\le g_x(y) \le C\abs{x-y}^{2-n},\quad \forall y \neq x,
\]
where $C$ depends only on $n$, $p$, $\lambda$, $\norm{\mathbf A}_{\infty}$, $\norm{\vec b-\vec c}_p$, and $\abs{\Omega}$.
\end{proposition}
\begin{proof}
Note that, from Proposition \ref{GreenConstructionLipschitz}, we have
\[
\norm{g_x^k}_{L^{\frac{n}{n-2},\infty}(\Omega)}		\le C,
\]
where $C$ depends only on $n$, $\lambda$, $\norm{\mathbf A}_{\infty}$, $\abs{\Omega}$, $r_{\vec b-\vec c}\left(\frac{\lambda}{3C_n}\right)$, and $\tilde{r}_{\vec b-\vec c}\left(\frac{\lambda}{3C_n}\right)$.
Since $p>n$, by Lemmas \ref{RfBound} and \ref{RfBound2}, we find that $C$ depends only on $n$, $p$, $\lambda$, $\norm{\mathbf A}_{\infty}$, $\norm{\vec b-\vec c}_p$, and $\abs{\Omega}$.

Now, let $y\in\Omega$, set $r=\frac{1}{4} \abs{x-y}$, and consider $B_{2r}=B_{2r}(y)$.
We have the following two alternatives:
\begin{enumerate}[{\bf i.}]
\item	
$\boldsymbol{B_{2r} \subset \Omega:}$ 
Consider $k \ge k_0 >1/r$.
Then, the functions $g_x^k$ are $W^{1,2}(B_{2r})$ solutions of the equation $L\tran g_x^k=0$ in $B_{2r}$.
Hence, from Lemma~\ref{LocalBoundForp} and the embedding $L^1(B_r)\hookrightarrow L^{\frac{n}{n-2},\infty}(B_r)$, we obtain that
\begin{equation}\label{eq:NotForCritical!}
\sup_{B_{r/2}}g_x^k\le C\fint_{B_r}g_x^k\le\frac{C}{r^n} \norm{g_x^k}_{L^1(B_r)}\le\frac{C}{r^n}r^2 \norm{g_x^k}_{L^{\frac{n}{n-2},\infty}(\Omega)}\le Cr^{2-n},
\end{equation}
where $C$ depends only on $n$, $p$, $\lambda$, $\norm{\mathbf A}_{\infty}$, $\norm{\vec b-\vec c}_p$, and $\abs{\Omega}$.

\item
$\boldsymbol{B_{2r} \not\subset \Omega:}$
Consider the ``approximate'' Green's functions $\widetilde G_y^m$ for the operator
\[
\widetilde{L}u=-\dv(\mathbf{A}\nabla u)+(\vec c-\vec b) \cdot\nabla u;
\]
that is, let $\widetilde{G}_y^m\in W_0^{1,2}(\Omega)$ be the solution to the equation
\[
\widetilde{L}\widetilde{G}_y^m:=-\dv\left(\mathbf{A}\nabla\widetilde{G}_y^m\right)+(\vec c-\vec b)\cdot \nabla \widetilde{G}_y^m=\abs{B_{1/m}(y)}^{-1}\chi_{B_{1/m}(y)}=h_m\;\text{ in }\;\Omega,
\]
which exists by Proposition \ref{G_y}.
Then, $\widetilde{v}:=\widetilde{G}_y^m-G_y^m\in W_0^{1,2}(\Omega)$ and  we find
\begin{align*}
-\dv(\mathbf{A}\nabla\widetilde{v}+\vec b\widetilde{v})+\vec c \cdot\nabla\widetilde{v}+d\widetilde{v}&=-\dv(\mathbf{A}\nabla\widetilde{G}_y^m+\vec b\widetilde{G}_y^m)+\vec c \cdot\nabla\widetilde{G}_y^m+d\widetilde{G}_y^m-h^m\\
&=\widetilde{L}\widetilde{G}_y^m-\dv(\vec b\widetilde{G}_y^m)+\vec b\cdot\nabla\widetilde{G}_y^m+d\widetilde{G}_y^m-h_m\\
&=(d-\dv\vec b)\widetilde{G}_y^m.
\end{align*}
Since $d\ge\dv\vec b$ and $\widetilde{G}_y^m\ge 0$, the last quantity is nonnegative.
Therefore, the maximum principle (Proposition \ref{BoundForCriticalSubsolutions}) implies that $\widetilde{v}\ge 0$, and thus we have
\begin{equation}\label{eq:gxk1}
G_y^m\le\widetilde{G}_y^m.
\end{equation}
We now extend $\mathbf{A}$ to $\overline{\mathbf{A}}$, which is equal to $\lambda\mathbf{I}$ outside $\Omega$.
Moreover, we extend $\vec c-\vec b$ to $\overline{\vec c}\in\Lip(\mathbb R^n)$, with
\[
\norm{\overline{\vec c}}_{L^p(\mathbb R^n)} \le 2 \norm{\vec b-\vec c}_p.
\]
Consider now the functions $\overline G_y^m\in W_0^{1,2}(\Omega \cup B_{2r})$, that solve the equations
\[
\overline{L}\overline{G}_y^m:=-\dv\left(\overline{\mathbf{A}}\nabla\overline{G}_y^m\right)+\overline{\vec c} \cdot \nabla\overline{G}_y^m=\abs{B_{1/m}(y)}^{-1}\chi_{B_{1/m}(y)}=h_m\;\text{ in }\; \Omega\cup B_{2r}.
\]
Then, the function $\overline{v}=\overline{G}_y^m-\widetilde{G}_y^m$ solves the equation $\widetilde{L}\overline{v}=0$ in $\Omega$.
Since $\widetilde{G}_y^m$ vanishes on $\partial\Omega$ and $\overline{G}_y^m\ge 0$ in $\Omega\cup B_{2r}$, we see that $\overline{v}\ge 0$ on $\partial\Omega$.
Hence, the maximum principle shows that $\overline{v}\ge 0$ in $\Omega$ and thus, we have $\widetilde{G}_y^m\le\overline{G}_y^m$ in $\Omega$.
Combining with \eqref{eq:gxk1}, we obtain that
\begin{equation}\label{eq:overlineG}
G_y^m\le\overline{G}_y^m \quad\text{in}\quad \Omega.
\end{equation}
Since $B_{2r}\subset \Omega\cup B_{2r}$, we are back to the first case:
if $\overline{g}_x^k$ solves the equation
\[
-\dv(\overline{\mathbf{A}}\tran\nabla\overline{g}_x^k+\overline{\vec c}\,\overline{g}_x^k)=\abs{B_{1/k}(x)}^{-1}\chi_{B_{1/k}(x)}\quad\text{in}\quad \Omega \cup B_{2r},
\]
then we have
\[
\sup_{B_{r/2}}\,\overline{g}_x^k\le Cr^{2-n},
\]
where $C$ depends only on $n$, $p$, $\lambda$, $\norm{\vec b-\vec c}_p$, and $\abs{\Omega}$.
Therefore, using Lemma~\ref{Presymmetry} and \eqref{eq:overlineG}, we obtain
\[
g_x^k(y)=\lim_{m\to\infty}\fint_{B_{1/k}(x)}G_y^m\le\lim_{m\to\infty}\fint_{B_{1/k}(x)}\overline{G}_y^m=\overline{g}_x^k(y)\le Cr^{2-n}.
\]
\end{enumerate}
Therefore, in both cases, we have $g_x^k(y)\le Cr^{2-n}$.
Since a subsequence of $g_x^k$ converges to $g_x$ almost everywhere in $\Omega$, we get
\[
g_x(y)\le Cr^{2-n}\le C\abs{x-y}^{2-n}.\qedhere
\]
\end{proof}

We also show the next lemma, which will be useful to us later.

\begin{lemma}\label{W^{1,2}Outside}
Let $\Omega\subset \mathbb R^n$ be a domain with $\abs{\Omega} <+\infty$.
Let $\mathbf{A}$ be bounded and satisfy the uniform ellipticity condition \eqref{ellipticity} and $\vec b$, $\vec c$, $d\in \Lip(\Omega)\cap L^{\infty}(\Omega)$, and $d \ge \dv \vec b$.
Let $p>n$ and $r>0$. Then, for almost every $x\in\Omega$ and all $k\in\mathbb N$, the functions $g_x^k$ belong to $W^{1,2}(\Omega\setminus B_r(x))$ and $\norm{g_x^k}_{W^{1,2}(\Omega\setminus B_r(x))}$ are bounded uniformly in $k$.
\end{lemma}
\begin{proof}
From Proposition~\ref{GreenConstructionLipschitzSubcritical}, we have $0\le g_x^k(y)\le C\abs{x-y}^{2-n}$, uniformly in $k$. This shows that $g_x^k$ is uniformly bounded in $\Omega\setminus B_r(x)$.
Hence, using Lemma~\ref{Caccioppoli}, we obtain that $g_x^k\in W^{1,2}(\Omega\setminus B_r(x))$ and $\norm{g_x^k}_{W^{1,2}(\Omega\setminus B_r(x))}$ are bounded uniformly in $k$.
\end{proof}

As a corollary, we also obtain the next estimate; we will be able to weaken the assumptions later.

\begin{corollary}\label{TSBound}
Let $\Omega$ be a domain with $\abs{\Omega} <+\infty$.
Let $\mathbf{A}$ be bounded and satisfy the uniform ellipticity condition \eqref{ellipticity} and $\vec b$, $\vec c$, $d\in \Lip(\Omega)\cap L^{\infty}(\Omega)$, and $d\ge\dv\vec b$.
Then the operators
\[
T, S:W^{-1,2}(\Omega)\to W_0^{1,2}(\Omega),
\]
defined in \eqref{eq:T}, \eqref{eq:S}, are bounded.
Moreover, for $p>n$, the operator norms of $T$ and $S$ are bounded by a constant depending only on $n$, $p$, $\lambda$, $\norm{\mathbf A}_{\infty}$, $\norm{\vec b-\vec c}_p$, and $\abs{\Omega}$.
\end{corollary}
\begin{proof}
It suffices to show that $T$ is bounded, since $S=T^*$.
To show this, note that, from Lemma \ref{Gamma}, there exists $\gamma>0$  such that the bilinear form
\[
\alpha_{\gamma}(u,v)=\int_{\Omega}\mathbf{A}\nabla u \cdot\nabla v+\vec c\cdot \nabla v\, u+\vec b \cdot \nabla u\, v+(d+\gamma)uv
\]
is coercive. Therefore, for every $F\in W^{-1,2}(\Omega)$, there exists $u\in W_0^{1,2}(\Omega)$ such that
\[
\alpha_{\gamma}(u,v)=\ip{F,v},\quad \forall v \in W_0^{1,2}(\Omega).
\]
Moreover, by Lemma~\ref{Gamma}, we have that
\[
\frac{\lambda}{3}\int_{\Omega}|\nabla u|^2\leq\alpha_{\gamma}(u,u)=\ip{F,u}\leq\norm{F}_{W^{-1,2}(\Omega)}\norm{u}_{W_0^{1,2}(\Omega)},
\]
and the Sobolev inequality shows that $\norm{u}_{W_0^{1,2}(\Omega)}\le C\norm{F}_{W^{-1,2}(\Omega)}$, where $C$ depends only on $n$, $\lambda$, and $\abs{\Omega}$.

Let a sequence $\set{u_n} \in C_c^{\infty}(\Omega)$ converge to $u$ in $W_0^{1,2}(\Omega)$, and set
\[
v_n(y)=\gamma\int_{\Omega}g_y(x)u_n(x)\,dx.
\]
Then, since $u_n\in L^{\infty}(\Omega)$ for all $n$, Proposition \ref{g_x} shows that $v_n$ is the solution to the equation $Lv_n=\gamma u_n$. Consider also the solution $v\in W_0^{1,2}(\Omega)$ to the equation $Lv=\gamma u$. Then, from \eqref{eq:T}, $v_n=T(\gamma u_n)$ and $v=T(\gamma u)$, therefore continuity of $T$ shows that
\begin{equation}\label{eq:BoundOnV}
\|v\|_{W_0^{1,2}(\Omega)}=\|T(\gamma u)\|_{W_0^{1,2}(\Omega)}=\lim_{n\to\infty}\|T(\gamma u_n)\|_{W_0^{1,2}(\Omega)}=\lim_{n\to\infty}\|v_n\|_{W_0^{1,2}(\Omega)}.
\end{equation}
Now, from Proposition~\ref{GreenConstructionLipschitzSubcritical} and the Cauchy-Schwartz inequality, we obtain that
\[
\abs{v_n(y)} \le C\gamma\int_{\Omega} \abs{x-y}^{2-n} \abs{u_n(x)}\,dx\le C\left(\int_{\Omega} \abs{x-y}^{2-n}\,\abs{u_n(x)}^2\,dx\right)^{1/2},
\]
where $C$ depends (via Lemma \ref{RfBound}) only on $n$, $p$, $\lambda$, $\norm{\mathbf A}_{\infty}$, $\norm{\vec b-\vec c}_p$, and $\abs{\Omega}$.
Here, we used (see the calculation before \cite[Lemma 7.12]{Gilbarg})
\[
\int_\Omega \abs{x-y}^{2-n} \,dx \le \frac{n}{2}\, \abs{B_1}^{\frac{n-2}{n}} \abs{\Omega}^{\frac{2}{n}} = C(n)\, \abs{\Omega}^{\frac{2}{n}}.
\]
Therefore, by Fubini's theorem, we get
\begin{equation}\label{eq:w}
\int_{\Omega} \abs{v_n(y)}^2\,dy \le C\int_{\Omega}\int_{\Omega} \abs{x-y}^{2-n}\, \abs{u_n(x)}^2\,dy\,dx\le C\int_{\Omega} \abs{u_n(x)}^2\,dx.
\end{equation}
Moreover, Lemma~\ref{BoundOnU} shows that (recall $Lv_n=\gamma u_n$)
\[
\int_{\Omega} \abs{\nabla v_n}^2\le C\int_{\Omega} \abs{\gamma u_n}^2+C\int_{\Omega} \abs{v_n}^2\le C\int_{\Omega} \abs{u_n}^2,
\]
where we also used \eqref{eq:w}. Combining with \eqref{eq:w}, we obtain that
\[
\norm{v_n}_{W_0^{1,2}(\Omega)}\le C\norm{u_n}_{L^2(\Omega)},
\]
and thus, plugging in \eqref{eq:BoundOnV}, we finally find that
\[
\norm{v}_{W_0^{1,2}(\Omega)}=\lim_{n\to\infty} \norm{v_n}_{W_0^{1,2}(\Omega)}\leq C\lim_{n\to\infty} \norm{u_n}_{L^2(\Omega)}=C \norm{u}_{L^2(\Omega)}.
\]
Finally, setting $w=u-v$, we see that $w$ is the solution to the equation $Lw=F$ and
\[
\norm{w}_{W_0^{1,2}(\Omega)}\le \norm{u}_{W_0^{1,2}(\Omega)}+\norm{v}_{W_0^{1,2}(\Omega)}\le C\norm{u}_{W_0^{1,2}(\Omega)}\le C\norm{F}_{W^{-1,2}(\Omega)},
\]
where $C$ depends only on $n$, $p$, $\lambda$, $\norm{\mathbf A}_{\infty}$, $\norm{\vec b-\vec c}_p$, and $\abs{\Omega}$.
\end{proof}

\subsection{The good estimates}

Using an approximation argument, we drop the assumption $\vec b$, $\vec c$, $d \in \Lip(\Omega)$.
We first show the next lemma.

\begin{lemma}\label{DivergenceApproximation}
Let $\vec b\in L^n(\Omega)$, $d\in L^{n/2}(\Omega)$, with $d\ge\dv\vec b$.
Let $\psi_m(x)=m^n\psi(mx)$ be a mollifier, with $\psi$ supported in the unit ball, and let
\[
\vec b_m=(\vec b \chi_\Omega) \ast \psi_m,\quad d_m=(d \chi_\Omega) \ast \psi_m.
\]
Then $\vec b_m,d_m\in\Lip(\Omega_m)\cap L^{\infty}(\Omega_m)$ and $d_m\ge\dv\vec b_m$ in $\Omega_m:=\Set{x\in\Omega: {\rm dist}(x,\partial\Omega)>\frac{1}{m}}$.
\end{lemma}
	
\begin{proof}
To show that $d_m \in L^\infty(\Omega_m)$, we estimate for $x \in\Omega_m$ that
\[
\abs{d_m(x)}=\Abs{\int_{B_{1/m}(x)} d(z)\psi_m(x-z)\,dz} \le \norm{\psi_m}_\infty \int_{B_{1/m}(x)} \abs{d} \le  \norm{\psi_m}_\infty \norm{d}_{n/2} \abs{B_{1/m}}^{1-2/n}\leq K_m \norm{d}_{n/2},
\]
where $K_m$ is a constant depending on $m$.
To see that $d_m \in\Lip(\Omega_m)$, note that  for $x$, $y\in\Omega_m$, we have
\begin{align*}
\abs{d_m(x)-d_m(y)}&=\Abs{\int_{B_{1/m}(x)\cup B_{1/m}(y)} d(z)\left(\psi_m(x-z)-\psi_m(y-z)\right)dz}\\
&\leq \norm{\nabla \psi_m}_\infty \abs{x-y} \int_{B_{1/m}(x)\cup B_{1/m}(y)} \abs{d} \le K_m \norm{d}_{n/2} \abs{x-y}.
\end{align*}
Similarly, we obtain that $\vec b_m \in\Lip(\Omega_m) \cap L^\infty(\Omega_m)$.

Let $\phi\in C_c^{\infty}(\Omega_m)$ be nonnegative, and set $\phi_y(z)=\phi(z+y)$. We then compute
\begin{align*}
\int_{\Omega_m}d_m\phi+\vec b_m \cdot\nabla\phi&=\int_{\Omega_m}\left(\int_{B_{1/m}}\left(d(x-y)\psi_m(y)\phi(x)+\vec b(x-y)\psi_m(y) \cdot \nabla\phi(x)\right)\,dy\right)\,dx\\
&=\int_{B_{1/m}}\left(\int_{\Omega_m}\left(d(x-y)\phi(x)+\vec b(x-y) \cdot\nabla\phi(x)\right)\,dx\right)\psi_m(y)\,dy\\
&=\int_{B_{1/m}}\left(\int_{\Omega_m-y}\left(d(z)\phi_y(z)+\vec b(z) \cdot\nabla\phi_y(z)\right)\,dz\right)\psi_m(y)\,dy.
\end{align*}
Since $\phi\in C_c^{\infty}(\Omega_m)$, we obtain that $\phi_y\in C_c^{\infty}(\Omega_m-y)$.
Moreover, $\Omega_m-y\subset\Omega$ whenever $y\in B_{1/m}$.
Therefore, if we extend $\phi_y$ by zero outside $\Omega_m-y$, we obtain that $\phi_y\in C_c^{\infty}(\Omega)$, therefore the last inner integral is equal to
\[
\int_{\Omega_m-y}\left(d(z)\phi_y(z)+\vec b(z) \cdot\nabla\phi_y(z)\right)\,dz=\int_{\Omega}\left(d(z)\phi_y(z)+\vec b(z) \cdot \nabla\phi_y(z)\right)\,dz\ge 0,
\]
since $d\ge\dv\vec b$ in $\Omega$.
Therefore, we obtain that
\[
\int_{\Omega_m}d_m\phi+\vec b_m \cdot \nabla\phi\ge 0,
\]
hence $d_m\ge\dv\vec b_m$ in $\Omega_m$.
\end{proof}

\begin{theorem}\label{GreenConstructionAdjoint}
Let $\Omega\subset \mathbb R^n$ be a domain with $\abs{\Omega} <+\infty$.
Let $\mathbf{A}$ be bounded and satisfy the uniform ellipticity condition \eqref{ellipticity} and $\vec b$, $\vec c\in L^n(\Omega)$, $\vec b-\vec c \in L^p(\Omega)$ for some $p>n$, $d\in L^{n/2}(\Omega)$. Assume also that $d\ge\dv\vec b$. For almost every $x\in\Omega$, the function $g_x$ in Proposition \ref{g_x}   satisfies the bounds
\[
0 \le g_x(y)\le C \abs{x-y}^{2-n},\quad\text{for }\;x \neq y,
\]
where $C>0$ depends on $n$, $p$, $\lambda$, $\norm{\mathbf A}_{\infty}$, $\norm{\vec b-\vec c}_p$, and $\abs{\Omega}$.
\end{theorem}

\begin{proof}
Let $\vec b_m$, $\vec c_m$, $d_m$ be mollifications of $\vec b$, $\vec c$, $d$, as in Lemma \ref{DivergenceApproximation}.
Consider the operators
\[
L_m\tran=-\dv(\mathbf{A}\tran\nabla u+\vec c_mu)+\vec b_m \cdot \nabla u+d_mu \quad \text{in}\quad \Omega_m
\]
as in Lemma~\ref{DivergenceApproximation}, then $d_m\ge\dv\vec b_m$ in $\Omega_m$.

Let $g_x^m$ be Green's function for the operator $L_m\tran$ in $\Omega_m$, as in Proposition~\ref{g_x}.
Note that $\abs{\Omega_m}\leq \abs{\Omega}$, and also $\norm{\vec b_m-\vec c_m}_{L^p(\Omega_m)} \le \norm{\vec b-\vec c}_{L^p(\Omega)}$.
Moreover, by the proof of Lemma~\ref{DivergenceApproximation} we have that $\vec b_m,\vec c_m, d_m\in\Lip(\Omega_m)\cap L^{\infty}(\Omega_m)$, therefore Proposition~\ref{GreenConstructionLipschitzSubcritical} is applicable, and we have
\begin{equation}	\label{eq:gxm}
g_x^m(y)\le C \abs{x-y}^{2-n},
\end{equation}
where $C$ depends only on $n$, $p$, $\lambda$, $\norm{\mathbf A}_{\infty}$, $\norm{\vec b_m-\vec c_m}_p$, and $\abs{\Omega_m}$. Hence, $C$ depends only on $n$, $p$, $\lambda$, $\norm{\mathbf A}_{\infty}$, $\norm{\vec b-\vec c}_p$, and $\abs{\Omega}$; in particular, it can be taken independent of $m$.
By extending $g_x^m$ by zero outside $\Omega_m$, we can assume that $g_x^m\in W_0^{1,q}(\Omega)$ for some $q\in\left(1,\frac{n}{n-1}\right)$ and by Proposition~\ref{g_x}, we have
\[
\norm{g_x^m}_{W_0^{1,q}(\Omega_m)}=\norm{g_x^m}_{W_0^{1,q}(\Omega)}\le C_q
\]
uniformly in $x$.
Therefore there exists a subsequence $\set{g_x^{k_m}}$ that converges to a function $g_x\in W_0^{1,q}(\Omega)$, weakly in $W_0^{1,q}(\Omega)$, strongly in $L^q(\Omega)$, and almost everywhere in $\Omega$.
In particular, from \eqref{eq:gxm}, we obtain that 
\[
g_x(y)=\lim_{m\to\infty}g_x^{k_m}(y)\le C\abs{x-y}^{2-n}.
\]
Finally, to show that $g_x$ is Green's function for $L\tran$ in $\Omega$, we let $S_m$ be the operator that sends $F\in W^{-1,2}(\Omega_m)$ to the solution $u_m\in W_0^{1,2}(\Omega_m)$ of $L_m\tran u_m=F$ in $\Omega_m$ and let $S$ be the corresponding operator for $L\tran$ in $\Omega$.
Then, for $\phi\in L^{\infty}(\Omega)$, we have
\[
\int_{\Omega}g_x\phi=\lim_{m\to\infty}\int_{\Omega}g_x^{k_m}\phi=\lim_{m\to\infty}\int_{\Omega_{k_m}}g_x^{k_m}\phi=\lim_{m\to\infty}S_{k_m}\phi(x)
\]
for almost every $x$.
Note that
\[
\norm{S_m\phi}_{W_0^{1,2}(\Omega)}=\norm{S_m\phi}_{W_0^{1,2}(\Omega_m)}\le \norm{S_m} \, \norm{\phi}_{W_0^{1,2}(\Omega)}\le C \norm{\phi}_{W_0^{1,2}(\Omega)},
\]
independent of $m$, where we used Corollary~\ref{TSBound}. Therefore $\set{S_{k_m}\phi}$ has a subsequence, still denoted by $\set{S_{k_m}\phi}$ that converges weakly and almost everywhere to some $u\in W_0^{1,2}(\Omega)$. If now $\psi\in C_c^{\infty}(\Omega)$, then $\psi\in C_c^{\infty}(\Omega_m)$ for $m$ large enough, therefore
\[
\int_{\Omega_m}\mathbf{A}\tran\nabla S_m\phi \cdot\nabla\psi+\vec c_m \cdot \nabla\psi\, S_m\phi+\vec b_m \cdot \nabla S_m\phi\,\psi+d_mS_m\phi\,\psi=\int_{\Omega_m}\phi\,\psi.
\]
Hence, we compute
\[
\int_{\Omega}\phi\psi=\lim_{m\to\infty}\int_{\Omega_{k_m}}\phi\psi=\int_{\Omega}\mathbf{A}\tran\nabla u \cdot \nabla\phi+\vec c \cdot \nabla\psi\,  u+\vec b \cdot \nabla u\,\psi+du\psi,
\]
where we used $b_{k_m}\to b$, $c_{k_m}\to c$ in $L^p$, $d_{k_m}\to d$ in $L^{n/2}$, and $S_{k_m}\phi \rightharpoonup u$ weakly in $W_0^{1,2}(\Omega)$. Therefore, $S_{k_m}\phi$ converges almost everywhere to the $W^{1,2}_0$ solution of $Lu=\phi$ in $\Omega$, which completes the proof.
\end{proof}

In Section~\ref{sec7}, we will show that $\nabla g_x\in L^{\frac{n}{n-1},\infty}(\Omega)$, even in the critical case where $\vec b,\vec c\in L^n(\Omega)$ and $d\in L^{n/2}(\Omega)$.

Finally, we show the analogous results for the function $G_y$.

\begin{theorem}\label{GreenConstruction}
Let $\Omega\subset\mathbb R^n$ be a domain with $\abs{\Omega} <+\infty$.
Let $\mathbf{A}$ be bounded and satisfy the uniform ellipticity condition \eqref{ellipticity} and $\vec b$, $\vec c\in L^n(\Omega)$, $\vec b-\vec c \in L^p(\Omega)$ for some $p>n$, $d\in L^{n/2}(\Omega)$, with $d\ge \dv\vec b$.
Then there exists a function $G(\cdot, \cdot)$ on $\Omega \times \Omega$
such that the weak solution $u\in W_0^{1,2}(\Omega)$ of the equation $L\tran u=\phi$, for $\phi\in L^{\infty}(\Omega)$, is given by
\[
u(y)=\int_{\Omega}G(x,y)\phi(x)\,dx.
\]
Moreover, we have
\[
\norm{\nabla G(\cdot,y)}_{L^{\frac{n}{n-1},\infty}(\Omega)}\le C,
\]
and for every $x,y\in\Omega$ with $x\neq y$, we have
\[
G(x,y)\le C \abs{x-y}^{2-n},
\]
where $C$ depends only on $n$, $p$, $\lambda$, $\norm{\mathbf A}_{\infty}$, $\norm{\vec b-\vec c}_p$, and $\abs{\Omega}$.
\end{theorem}
\begin{proof}
We first assume that $\vec b$, $\vec c$, $d\in \Lip(\Omega)\cap L^{\infty}(\Omega)$, and we will show that
\[
\norm{\nabla G_y}_{L^{\frac{n}{n-1},\infty}(\Omega)}\le C,
\]
where $C$ depends on $n$, $p$, $\lambda$, $\norm{\mathbf A}_{\infty}$, $\norm{\vec b-\vec c}_p$, and $\abs{\Omega}$. Note first that, from the first relation in Lemma~\ref{Presymmetry} and Proposition~\ref{GreenConstructionLipschitzSubcritical}, we have
\[
\abs{G_y(x)} \le C\abs{x-y}^{2-n}.
\]
Using a procedure similar to Lemma~\ref{W^{1,2}Outside}, we can show that $G_y$ is a $W^{1,2}(\Omega\setminus B_{\tau}(y))$ solution of the equation $Lu=0$ in $\Omega\setminus B_{\tau}(y)$ provided the last set is nonempty. Moreover, by construction, $G_y\in W_0^{1,q}(\Omega)$, for $1<q<\frac{n}{n-1}$. Therefore the conditions of Proposition~\ref{GeneralLorentzBound} are satisfied, hence
\[
\norm{\nabla G_y}_{L^{\frac{n}{n-1},\infty}(\Omega)}\le C\norm{G_y}_{L^{\frac{n}{n-2},\infty}(\Omega)}^2+C\leq C,
\]
where $C$ depends on $n$, $\lambda$, $\norm{\vec b-\vec c}_p$, and $\abs{\Omega}$. This completes the proof in the case where $\vec b,\vec c,d$ are bounded Lipschitz functions.

In the general case, we use an approximation argument as in Theorem ~\ref{GreenConstructionAdjoint}.
Let $\vec b_m$, $\vec c_m$, $d_m$, and $\Omega_m$ be as in the proof of Theorem~\ref{GreenConstructionAdjoint}.
Consider the operator
\[
L_m u=-\dv (\mathbf{A}\nabla u+\vec b_m u)+\vec c_m\cdot\nabla u+d_m u \quad\text{in}\quad \Omega_m
\]
and let $G_y^m$ be Green's function for $L_m$ in $\Omega_m$.
Then, by the above estimate and the pointwise bound, we find that
\[
\norm{G_y^m}_{L^{\frac{n}{n-2},\infty}(\Omega)}\le C\quad\text{and}\quad \norm{\nabla G_y^m}_{L^{\frac{n}{n-1},\infty}(\Omega)}\le C,
\]
for some $C$ that depends only on $n$, $\lambda$, $\norm{\mathbf A}_{\infty}$, $\norm{\vec b-\vec c}_p$, and $\abs{\Omega}$; the fact that $C$ does not depend on $m$ follows from an argument similar to the one in the proof of Theorem~\ref{GreenConstructionAdjoint}.
Since the space $L^{\frac{n}{n-1},\infty}(\Omega)$ is the dual of $L^{n,1}(\Omega)$ and $L^{\frac{n}{n-2},\infty}(\Omega)$ is the dual of $L^{\frac{n}{2},1}(\Omega)$, the Banach-Alaoglu theorem implies that there exists $\tilde{G}_y\in L^{\frac{n}{n-2},\infty}(\Omega)$, with $\nabla\tilde{G}_y\in L^{\frac{n}{n-1},\infty}(\Omega)$, and a subsequence $G_y^{i_m}$ such that
\[
G_y^{i_m}\rightharpoonup \tilde{G}_y\;\text{ weak* \, in }\;L^{\frac{n}{n-2},\infty}(\Omega)\quad\text{and}\quad \nabla G_y^{i_m}\rightharpoonup \nabla\tilde{G}_y\;\text{ weak*  \,in }\; L^{\frac{n}{n-1},\infty}(\Omega).
\]
Therefore, we have
\[
\norm{\tilde{G}_y}_{L^{\frac{n}{n-2},\infty}(\Omega)}\le\limsup_{m\to\infty}\, \norm{G_y^{i_m}}_{L^{\frac{n}{n-2},\infty}(\Omega)}\le C,
\]
and also
\[
\norm{\nabla\tilde{G}_y}_{L^{\frac{n}{n-1},\infty}(\Omega)}\le\limsup_{m\to\infty}\, \norm{\nabla G_y^{i_m}}_{L^{\frac{n}{n-1},\infty}(\Omega)}\le C.
\]
Moreover, using an argument similar to the proof of Theorem~\ref{GreenConstructionAdjoint}, we can show that a subsequence of $G_y^{i_m}$ converges to $G_y$ weakly in $W_0^{1,q}(\Omega)$ for all $q\in (1,\frac{n}{n-1})$.
This implies that $G_y=\tilde{G}_y$, which completes the proof.
\end{proof}

We now show the next symmetry relation between $g_x$ and $G_y$.

\begin{proposition}\label{Symmetry}
Let $\Omega\subset \mathbb R^n$ be a domain with $\abs{\Omega} <+\infty$.
Let $\mathbf{A}$ be bounded and satisfy the uniform ellipticity condition \eqref{ellipticity} and $\vec b$, $\vec c \in L^n(\Omega)$, $\vec b-\vec c \in L^p(\Omega)$ for $p>n$, $d\in L^{n/2}(\Omega)$, and $d\ge\dv\vec b$. Then, the functions $g$ and $G$ from Theorems \ref{GreenConstructionAdjoint} and \ref{GreenConstruction} satisfy the relation
\[
G(x,y)=g(y,x)
\]
for almost all $x,y\in\Omega$ with $x\neq y$.
\end{proposition}
\begin{proof}
In the case where $\vec b$, $\vec c$, $d\in \Lip(\Omega)\cap L^{\infty}(\Omega)$, we mimic the proof of Theorem 1.3 in \cite{Gruter}, using the first relation in Lemma~\ref{Presymmetry} and Proposition 7.1 in \cite{StampacchiaDirichlet}.
Then, for the general case, we use an approximation argument as in Theorem~\ref{GreenConstructionAdjoint}.
\end{proof}

We finish this section with the next proposition.

\begin{proposition}
Let $\Omega\subset \mathbb R^n$ be a domain with $\abs{\Omega} <+\infty$.
Let $\mathbf{A}$ be bounded and satisfy the uniform ellipticity condition \eqref{ellipticity} and $\vec b$, $\vec c\in L^n(\Omega)$, $\vec b-\vec c \in L^p(\Omega)$ for $p>n$, $d\in L^{n/2}(\Omega)$, and $d\ge \dv\vec b$. Then the operators
\[
T, \,S:W^{-1,2}(\Omega)\to W_0^{1,2}(\Omega),
\]
defined in \eqref{eq:T} and \eqref{eq:S}, are bounded.
Moreover, the operator norms $\norm{T}$ and $\norm{S}$ are bounded by a constant depending on $n$, $p$, $\lambda$, $\norm{\mathbf A}_{\infty}$, $\norm{\vec b-\vec c}_p$, and $\abs{\Omega}$.
\end{proposition}
\begin{proof}
The proof is identical to the proof of Corollary~\ref{TSBound}, using Theorem~\ref{GreenConstructionAdjoint}.
\end{proof}

\section{The critical case}			\label{sec7}	

At this point we return to the critical case.
We show that the gradient of Green's function for the adjoint equation belongs to the Lorentz space $L^{\frac{n}{n-1},\infty}(\Omega)$ using our standard assumptions.
Moreover, strengthening our assumptions, we will deduce a pointwise bound.

\subsection{The Lorentz bounds}

\begin{proposition}\label{PreliminaryLorentz}
Let $\Omega\subset \mathbb R^n$ be a domain with $\abs{\Omega}<+\infty$.
Let $\mathbf{A}$ be bounded and satisfy the uniform ellipticity condition \eqref{ellipticity} and $\vec b$, $\vec c$, $d\in \Lip(\Omega)\cap L^{\infty}(\Omega)$, and $d \ge \dv \vec b$.
Then, for almost every $x\in\Omega$, we have
\[
\norm{\nabla g_x}_{L^{\frac{n}{n-1},\infty}(\Omega)}\le C,	
\]
where $C$ depends on $n$, $\lambda$, $\norm{\mathbf A}_{\infty}$, $r_{\vec b-\vec c}\left(\frac{\lambda}{3C_n}\right)$, $\tilde{r}_{\vec b-\vec c}\left(\frac{\lambda}{3C_n}\right)$, and $\abs{\Omega}$.
\end{proposition}
\begin{proof}
The proof is similar to the proof of Theorem~\ref{GreenConstruction}. Since $\vec b$, $\vec c$, $d$ are Lipschitz and bounded, Lemma~\ref{W^{1,2}Outside} shows that for any $\tau>0$ fixed the functions $g_x^k$ belong to $W^{1,2}(\Omega\setminus B_{\tau})$, where $B_\tau=B_\tau(x)$, uniformly in $k$.
Moreover, $g_x^k$ satisfies the equation $L\tran g_x^k=0$ in $\Omega\setminus B_{\tau}$. Since a subsequence of $g_x^k$ converges to $g_x$ weakly in $W_0^{1,q}(\Omega)$ for $1<q <\frac{n}{n-1}$, we find that $g_x\in W^{1,2}(\Omega\setminus B_{\tau})$ and $g_x^k$ is a solution to the equation $L\tran g_x^k=0$ in $\Omega\setminus B_{\tau}$. Therefore, the conditions of Proposition~\ref{GeneralLorentzBound} are satisfied, hence
\[
\|\nabla g_x\|_{L^{\frac{n}{n-1},\infty}(\Omega)}\leq C\|g_x\|_{L^{\frac{n}{n-2},\infty}(\Omega)}^2+C\leq C,
\]
where we also used Proposition~\ref{GreenConstructionLipschitz} in the last inequality.
This completes the proof.
\end{proof}

We can now show good Lorentz bounds for Green's function for the adjoint equation.

\begin{theorem}			\label{GoodLorentzForGreenAdoint}
Let $\Omega\subset \mathbb R^n$ be a domain with $\abs{\Omega}<+\infty$. 
Let $\mathbf{A}$ be bounded and satisfy the uniform ellipticity condition \eqref{ellipticity} and $\vec b,\vec c\in L^n(\Omega)$, $d\in L^{n/2}(\Omega)$, with $d \ge \dv \vec b$. Then, for almost every $x\in\Omega$, there exists a function $g_x(y)=g(y,x)$ such that, for any $\phi\in L^{\infty}(\Omega)$, the solution to the equation $Lu=\phi$ is given by
\[
u(y)=\int_{\Omega}g(y,x)\phi(x)\,dx.
\]
Moreover, the function $g_x$ satisfies the bounds
\[
\norm{g_x}_{L^{\frac{n}{n-2},\infty}(\Omega)}\le C \quad\text{and}\quad \norm{\nabla g_x}_{L^{\frac{n}{n-1},\infty}(\Omega)}\le C,
\]
where $C$ depends only on $n$, $\lambda$, $\norm{\mathbf A}_{\infty}$, $r_{\vec b-\vec c}\left(\frac{\lambda}{9C_n}\right)$, $\tilde{r}_{\vec b-\vec c}\left(\frac{\lambda}{9C_n}\right)$, and $\abs{\Omega}$.
\end{theorem}
\begin{proof}
Let $\vec b_m$, $\vec c_m$, $d_m$, $\Omega_m$, $L\tran_m$, and $g_x^m$ be as in the proof of Theorem~\ref{GreenConstructionAdjoint}. 
Since $(\vec b_m-\vec c_m)\chi_{\Omega_m}$ converges to $\vec b-\vec c$ in $L^n(\Omega)$, Lemmas \ref{UniformSplit} and \ref{UniformSplit2} show that, for a subsequence $\vec b_{k_m}-\vec c_{k_m}$, we have
\[
\textstyle
r_{(\vec b_{k_m}-\vec c_{k_m})\chi_{\Omega_{k_m}}}\left(\frac{\lambda}{3C_n}\right)\le 2 r_{\vec b-\vec c}\left(\frac{\lambda}{9C_n}\right)\quad\text{and}\quad \tilde{r}_{(\vec b_{k_m}-\vec c_{k_m})\chi_{\Omega_{k_m}}}\left(\frac{\lambda}{3C_n}\right)\ge \tilde{r}_{\vec b-\vec c}\left(\frac{\lambda}{9C_n}\right).
\]
We note that the constants $C$ in Proposition~\ref{GreenConstructionLipschitz} and Proposition~\ref{PreliminaryLorentz} depend on the numbers $r_{\vec b-\vec c}\left(\frac{\lambda}{3C_n}\right)$ and $\tilde{r}_{\vec b-\vec c}\left(\frac{\lambda}{3C_n}\right)$ in such a way that they increase as  $r_{\vec b-\vec c}\left(\frac{\lambda}{3C_n}\right)$ increases and also as $\tilde{r}_{\vec b-\vec c}\left(\frac{\lambda}{3C_n}\right)$ decreases.
Therefore, by Proposition~\ref{GreenConstructionLipschitz} and Proposition~\ref{PreliminaryLorentz}, we see that
\[
\norm{g_x^m}_{L^{\frac{n}{n-2},\infty}(\Omega)}\le C,\quad
\norm{\nabla g_x^m}_{L^{\frac{n}{n-1},\infty}(\Omega)}\le C,
\]
for some $C$ that depends only on $n$, $\lambda$, $\norm{\mathbf A}_{\infty}$, $r_{\vec b-\vec c}\left(\frac{\lambda}{9C_n}\right)$, $\tilde{r}_{\vec b-\vec c}\left(\frac{\lambda}{9C_n}\right)$, and $\abs{\Omega}$.
Since $L^{\frac{n}{n-1},\infty}(\Omega)$ is the dual of $L^{n,1}(\Omega)$ and $L^{\frac{n}{n-2},\infty}(\Omega)$ is the dual of $L^{\frac{n}{2},1}(\Omega)$, the Banach-Alaoglu theorem implies that there exists $\tilde{g}_x\in L^{\frac{n}{n-2},\infty}(\Omega)$, with $\nabla\tilde{g}_x\in L^{\frac{n}{n-1},\infty}(\Omega)$, and a subsequence $g_x^{i_m}$ such that
\[
g_x^{i_m}\rightharpoonup\tilde{g}\;\text { weak*\, in }\; L^{\frac{n}{n-2},\infty}(\Omega),\quad \nabla g_x^{i_m}\rightharpoonup\nabla\tilde{g}\;\text{ weak*\, in }\;L^{\frac{n}{n-1},\infty}(\Omega).
\]
Therefore
\[
\norm{\tilde{g}_x}_{L^{\frac{n}{n-2},\infty}(\Omega)}\le\limsup_{m\to\infty}\, \norm{g_x^{i_m}}_{L^{\frac{n}{n-2},\infty}(\Omega)}\le C,
\]
and also
\[
\norm{\nabla \tilde{g}_x}_{L^{\frac{n}{n-1},\infty}(\Omega)}\le\limsup_{m\to\infty}\, \norm{\nabla g_x^{i_m}}_{L^{\frac{n}{n-1},\infty}(\Omega)}\le C.
\]
Moreover, using an argument similar to the proof of Theorem~\ref{GreenConstructionAdjoint}, we can show that a subsequence of $g_x^{i_m}$ converges to $g_x$ weakly in $W_0^{1,q}(\Omega)$, for all $q\in (1,\frac{n}{n-1})$. 
This implies that $g_x=\tilde{g}_x$, which completes the proof.
\end{proof}

\subsection{The pointwise bounds}
We now turn to the pointwise bounds for Green's function in the critical case. In order to obtain those bounds, we will need to strengthen our assumptions on the coefficients of the equation, since an analog of Theorem~\ref{GreenConstruction} for the adjoint equation does not hold in general; we now turn to this fact.

Consider the ball $B=B_{1/e}(0)\subset \mathbb R^n$ and set
\begin{equation}\label{eq:b}
\vec b(x)= - \frac{x}{\abs{x}^2\, \ln \abs{x}\,}
\end{equation}
for $x\in B$.
A simple calculation shows that $\vec b\in L^n(B)$, but $\vec b\notin L^p(B)$, for any $p>n$.

\begin{lemma}			\label{ToCounterexample}
Consider the function $\vec b$ in \eqref{eq:b}.
Then, for any $\delta>0$, there exists a bounded function $f_{\delta}: B\to\mathbb R$ such that, the solution $u_{\delta}\in W_0^{1,2}(B)$ to the equation
\[
-\Delta u_{\delta}-\dv(\delta\vec bu_{\delta})=f_{\delta}\quad\text{in}\quad B
\]
satisfies the estimate
\[
\abs{u_{\delta}(x)} \ge C_{\delta}\,\abs{ \ln \abs{x}}^{\delta}
\]
near $x=0$ for some $C_{\delta}>0$.
\end{lemma}
\begin{proof}
Set $v_{\delta}(x)=\abs{ \ln \abs{x}}^{\delta}$, for $\delta>0$. Then, simple calculations show that
\[
v_{\delta}\in W^{1,2}(B),\quad v_{\delta} \vert_{\partial B}=1,\quad-\Delta v_{\delta}-\dv(\delta\vec bv_{\delta})=0.
\]
We now set $w(x)=e^2 \abs{x}^2$, and $u_{\delta}=v_{\delta}-w$. Note then that $u_{\delta}\in W_0^{1,2}(\Omega)$, and
\[
\Delta w+\dv(\delta\vec bw)=2ne^2 -\frac{\delta ne^2}{\ln \abs{x}}+\frac{\delta e^2}{\ln^2 \abs{x}}=-f_{\delta}(x),
\]
where $f_{\delta}$ is bounded in $B$. Therefore, $u_{\delta}\in W_0^{1,2}(B)$ is the solution to the equation
\[
-\Delta u_{\delta}-\dv(\delta \vec bu_{\delta})=f_{\delta}.
\]
However,
\[
\abs{u_{\delta}(x)} \ge  \abs{\ln \abs{x}}^{\delta}-e^2 \abs{x}^2 \ge C_{\delta}\, \abs{\ln \abs{x}}^{\delta}\quad\text{near }\;x=0.	\qedhere
\]
\end{proof}

Note that the previous lemma shows that a solution $u\in W^{1,2}(\Omega)$ to the equation
\[
-\Delta u-\dv(\vec bu)=0
\]
is not necessarily bounded if we just assume that $\vec b\in L^n(\Omega)$, or even if we assume that $\vec b$ has small $L^n(\Omega)$ norm.
Hence, Lemma~\ref{LocalBoundForp} cannot hold if we just assume that $\vec c\in L^n(\Omega)$.

\begin{proposition}\label{Counterexample}
If Green's function $G(\cdot,y)=G_y(\cdot)$, for the operator $-\Delta u+\vec b\cdot \nabla u$ in $B$, exists, then it cannot belong to $L^1(B)$ uniformly in $y$. In particular, the bounds
\[
G(x,y)\le C\abs{x-y}^{2-n},\quad \|G(\cdot,y)\|_{L^{\frac{n}{n-2},\infty}(\Omega)}\leq C
\]
cannot hold.
\end{proposition}
\begin{proof}
Consider the functions $u$ and $f$ from Lemma~\ref{ToCounterexample} for $\delta=1$, and assume that $\norm{G_y}_1\le C$ uniformly in $y$.
Then, since $u$ solves the equation $-\Delta u-\dv(\vec bu)=0$ in $B$, we conclude from Definition~\ref{GreenDefinition} that
\[
\abs{u(y)}= \Abs{\int_B G(x,y) f(x)\,dx} \le \norm{G_y}_1\,\norm{f}_{\infty}\le C \norm{f}_{\infty}
\]
for almost every $y\in B$.
But then, by Lemma~\ref{ToCounterexample}, we have
\[
C_1 \abs{\ln \abs{y}} \le \abs{u(y)} \le C \norm{f}_{\infty}\quad\text{near }\;y=0,
\]
which is a contradiction.
\end{proof}

Using the solutions from Lemma~\ref{ToCounterexample} with $\delta>0$ small, we can show that Proposition~\ref{Counterexample} remains true, even if we assume that $b\in L^n(B)$ has arbitrarily small norm.

If we only assume that $d\ge\dv\vec b$, the point in which the proof of the bounds for the subcritical case cannot be adjusted to the critical case is the first estimate in \eqref{eq:NotForCritical!}. A way to make this estimate work is to assume that $d\ge\dv\vec c$ in the critical case. Then, the maximum principle holds for both $L$ and $L\tran$.

The following lemma is used to extend Lipschitz functions to larger domains. 

\begin{lemma}\label{GoodExtension}
Let $\Omega\subset \mathbb R^n$ be a domain, and let $f\in\Lip(\Omega)$. Let also $\varepsilon_1$, $\varepsilon_2>0$, and $r>0$.
Then, for any ball $B_r=B_r(x)$ with $x \in \Omega$, there exists a function $\overline{f}\in\Lip(\Omega\cup B_r)$, which is equal to $f$ in $\Omega$ and  satisfies
\[
r_{\overline{f}}(2\varepsilon_1)\le r_f(\varepsilon_1)\quad \text{and}\quad \tilde{r}_{\overline{f}}(2\varepsilon_2)\ge\tilde{r}_f(\varepsilon_2).
\]
Moreover, if $f$ is bounded, then $\overline{f}$ is also bounded.
\end{lemma}
\begin{proof}
By Theorem 3 in \cite[p. 174]{SteinSingular}, there exists a function $g\in\Lip(\mathbb R^n)$ which is equal to $f$ in $\Omega$ and $\abs{g}\le M:=\norm{f}_{\Lip(\Omega)}$.
Now, let $\delta>0$ and set
\[
E_{\delta}=\set{x\in B_r : 0<{\rm dist}(x,\Omega)<\delta}\quad \text{and}
\quad  F_{\delta}=\set{x\in B_r : {\rm dist}(x,\Omega)\ge\delta}.
\]
Note that if $x\in\overline{\Omega}$, then $\chi_{E_{\delta}}(x)=0$ and if $x\notin\overline{\Omega}$, then $\chi_{E_{\delta}}(x)=0$ for $\delta<\frac{1}{2}{\rm dist}(x,\Omega)$.
Therefore, $\chi_{E_{\delta}}\to 0$ pointwise in $B_r$ and thus, the dominated convergence theorem shows that
\[
\lim_{\delta \to 0} \, \abs{E_{\delta}} =0.
\]
Let us choose $\delta>0$ such that
\[
M^n \abs{E_{\delta}}<\varepsilon_1^n \quad \text{and} \quad  M^n \abs{E_{\delta}}<\varepsilon_2^n.
\]
Let $\phi$ be a smooth cutoff such that  $\phi=0$ in $ F_{\delta}$, $\phi \equiv 1$ in $\Omega$, and $0\le\phi\le 1$.
Set $\overline{f}=g\phi$.
Then, for every $t>0$, we compute
\[
\int_{\Omega\cup E_{\delta}} \Abs{\,\bigabs{\overline f}^{(t)}}^n=\int_{\Omega} \Abs{\,\bigabs{\overline f}^{(t)}}^n+\int_{E_{\delta}} \Abs{\,\bigabs{\overline f}^{(t)}}^n\le\int_{\Omega} \Abs{\,\abs{f}^{(t)}}^n+M^n \Abs{E_{\delta}}\le(R_f(t))^n+\varepsilon_1^n,
\]
from our choice of $\delta$, which shows that
\[
(R_{\overline{f}}(t))^n<(R_f(t))^n+\varepsilon_1^n.
\]
Hence, if $t<r_f(\varepsilon_1)$, then $R_{\overline{f}}(t)<2\varepsilon_1$, which shows the first claim.
	
For the second claim, let $t>0$ and $A\subseteq\overline{\Omega}\cup E_{\delta}$ with $\abs{A}<t$. Then,
\[
\int_A \bigabs{\overline f}^n=\int_{A\cap\Omega} \abs{f}^n+\int_{A\setminus\Omega} \bigabs{\overline f}^n\le(\tilde{R}_f(t))^n+\int_{E_{\delta}} \bigabs{\overline f}^n\le(\tilde{R}_f(t))^n+M^n \Abs{E_{\delta}}<(\tilde{R}_f(t))^n+\varepsilon_2^n,
\]
from our choice of $\delta$.
This shows that
\[
(\tilde{R}_{\overline{f}}(t))^n<(\tilde{R}_f(t))^n+\varepsilon_2^n,\quad \forall t>0.
\]
Hence, if $t<\tilde{r}_f(\varepsilon_2)$, then $\tilde{R}_{\overline{f}}(t)<2\varepsilon_2$, which shows the second claim.
It should be clear from the construction that $\overline{f}$ is also bounded if $f$ is bounded.
\end{proof}

\begin{proposition}\label{GreenConstructionLipschitzCritical}
Let $\Omega\subset \mathbb R^n$ be a domain with $\abs{\Omega} <+\infty$.
Let $\mathbf{A}$ be bounded and satisfy the uniform ellipticity condition \eqref{ellipticity} and $\vec b$, $\vec c$, $d\in\Lip(\Omega)\cap L^{\infty}(\Omega)$.
Assume also that $d\ge\dv\vec b$ and $d\ge\dv\vec c$. Then, the functions $g_x$ and $G_y$ satisfy the bounds
\[
 0\le g_x(z) \le C\abs{z-x}^{2-n}\quad\text{and}\quad 0 \le G_y(z) \le C\abs{z-y}^{2-n},
\]
where $C$ depends on $n$, $\lambda$, $\norm{\mathbf A}_{\infty}$, $r_{\vec b-\vec c}\left(\frac{\lambda}{6C_n}\right)$, $\tilde{r}_{\vec b-\vec c}\left(\frac{\lambda}{6C_n}\right)$, and $\abs{\Omega}$.
\end{proposition}

\begin{proof}
We mimic the proof of Proposition~\ref{GreenConstructionLipschitzSubcritical}.
In the case $B_{2r}=B_{2r}(x)\subset\Omega$, we use Lemma~\ref{LocalBoundForn} instead of Lemma~\ref{LocalBoundForp}, to get \eqref{eq:NotForCritical!}, which reads
\[
\sup_{B_{r/2} }g_x^k\le Cr^{2-n},
\]
where $C$ depends only on $n$, $\lambda$, $\norm{\mathbf A}_{\infty}$, $r_{\vec b-\vec c}\left(\frac{\lambda}{3C_n}\right)$, and $\abs{\Omega}$.

In the case when $B_{2r}(x)\not\subset \Omega$, we use Lemma~\ref{GoodExtension} to extend the function $\vec c-\vec b$ to a function $\overline{\vec c}\in\Lip(\Omega\cup B_{2r}(x))$, with
\[
\textstyle r_{\overline{\vec c}}\left(\frac{\lambda}{3C_n}\right)\le r_{\vec b-\vec c} \left(\frac{\lambda}{6C_n}\right)\quad\text{and}\quad \tilde{r}_{\overline{\vec c}}\left(\frac{\lambda}{3C_n}\right)\ge\tilde{r}_{\vec b-\vec c}\left(\frac{\lambda}{6C_n}\right).
\]
The rest of the proof runs without changes.

To show the bounds on $G_y^m$, note that our assumption $d\ge\dv\vec b$ makes the previous arguments applicable for $G_y^m$ in place of $g_x^k$.
\end{proof}

The next proposition is an analogue of Corollary~\ref{TSBound}.

\begin{proposition}\label{TSBoundCritical}
Let $\Omega \subset \bR^n$ be a domain with $\abs{\Omega} <+\infty$.
Let $\mathbf{A}$ be bounded and satisfy the uniform ellipticity condition \eqref{ellipticity} and $\vec b$, $\vec c\in L^n(\Omega)$, $d\in L^{n/2}(\Omega)$.
Also, assume that $d\ge\dv\vec b$ and $d\ge\dv\vec c$.
Then, the operators
\[
T,\,S:W^{-1,2}(\Omega)\to W_0^{1,2}(\Omega),
\]
defined in \eqref{eq:T} and  \eqref{eq:S}, are bounded.
Moreover, their norms are bounded by a constant that depends on $n$ and $\lambda$ only.
\end{proposition}
\begin{proof}
Note that our assumptions imply that, for any $u\in W_0^{1,2}(\Omega)$,
\[
\int_{\Omega}(\vec b+\vec c)\cdot\nabla u\,u+du^2=\frac{1}{2}\int_{\Omega}du^2+\vec b\cdot\nabla(u^2)+\frac{1}{2}\int_{\Omega}du^2+\vec c\cdot\nabla(u^2)\ge 0,
\]
therefore, if $u=TF$ for some $F\in W^{-1,2}(\Omega)$,
\[
\lambda\int_{\Omega} \abs{\nabla u}^2\le\int_{\Omega}\mathbf{A}\nabla u\cdot\nabla u+(\vec b+\vec c)\cdot\nabla u\,u+du^2=\ip{F,u}.
\]
We now apply the Sobolev inequality to conclude the proof.
\end{proof}

The next theorem is an analogue of Theorem~\ref{GreenConstructionAdjoint}.

\begin{theorem}\label{GreenConstructionAdjointCritical}
Let $\Omega\subset \mathbb R^n$ be a domain with $\abs{\Omega} <+\infty$.
Let $\mathbf{A}$ be bounded and satisfy the uniform ellipticity condition \eqref{ellipticity} and $\vec b$, $\vec c\in L^n(\Omega)$, $d\in L^{n/2}(\Omega)$.
Also, assume that $d\ge \dv\vec b$ and $d\ge \dv \vec c$.
Then there exists a function $G(\cdot, \cdot)$ on $\Omega\times \Omega$
such that the solution $u\in W_0^{1,2}(\Omega)$ of the equation $Lu=\phi$, for $\phi\in L^{\infty}(\Omega)$, is given by
\[
u(y)=\int_{\Omega}G(x,y)\phi(x)\,dx
\]
for almost every $y\in\Omega$. 
Moreover, we have
\[
\norm{\nabla G(\cdot,y)}_{L^{\frac{n}{n-1},\infty}(\Omega)}\le C,
\]
and for every $x,y\in\Omega$ with $x\neq y$, we have
\[
G(x,y)\le C \abs{x-y}^{2-n},
\]
where $C$ depends only on $n$, $\lambda$, $\norm{\mathbf A}_{\infty}$, $r_{\vec b-\vec c}\left(\frac{\lambda}{18C_n}\right)$, $\tilde{r}_{\vec b-\vec c}\left(\frac{\lambda}{18C_n}\right)$, and $\abs{\Omega}$.
\end{theorem}
\begin{proof}
Note first that the Lorentz bound follows from Theorem~\ref{GoodLorentzForGreenAdoint}.

For the proof of the pointwise bound, we follow the proof of Theorem~\ref{GreenConstructionAdjoint}.
We begin by considering mollifications $\vec b_m$, $\vec c_m$, $d_m$ of $\vec b$, $\vec c$, $d$, respectively.
Since $(\vec b_m-\vec c_m)\chi_{\Omega_m}\to\vec b-\vec c$ in $L^n(\Omega)$, Lemmas \ref{UniformSplit} and \ref{UniformSplit2} show that, for a subsequence $\vec c_{k_m}-\vec b_{k_m}$, we have
\[
\textstyle
r_{(\vec b_{k_m}-\vec c_{k_m})\chi_{\Omega_{k_m}}}\left(\frac{\lambda}{6C_n}\right)\le 2r_{\vec b-\vec c}\left(\frac{\lambda}{18C_n}\right)\quad \text{and}\quad \tilde{r}_{(\vec b_{k_m}-\vec c_{k_m})\chi_{\Omega_{k_m}}}\left(\frac{\lambda}{6C_n}\right)\ge \tilde{r}_{\vec b-\vec c}\left(\frac{\lambda}{18C_n}\right).
\]
Then, continuing as in the proof of Theorem~\ref{GreenConstructionAdjoint}, we use Proposition~\ref{GreenConstructionLipschitzCritical} instead of Proposition~\ref{GreenConstructionLipschitzSubcritical}, and Proposition~\ref{TSBoundCritical} instead of Corollary~\ref{TSBound} to complete the proof.
\end{proof}

\begin{remark}
If $L$ satisfies the assumptions of Theorem~\ref{GreenConstructionAdjointCritical}, then the adjoint operator $L\tran$ satisfies the same assumptions.
Hence, Theorem~\ref{GreenConstructionAdjointCritical} is also an analogue of Theorem~\ref{GreenConstruction}.
\end{remark}
Finally, by replicating the proof of Proposition~\ref{Symmetry}, we obtain the following proposition.

\begin{proposition}
Let $\Omega\subset \mathbb R^n$ be a domain with $\abs{\Omega} <+\infty$.
Let $\mathbf{A}$ be bounded and satisfy the uniform ellipticity condition \eqref{ellipticity} and $\vec b$, $\vec c\in L^n(\Omega)$, $d\in L^{n/2}(\Omega)$.
Assume also that $d\ge\dv\vec b$ and $d\ge\dv\vec c$.
If $g(\cdot, \cdot)$ is Green's function for the operator $L\tran$ in  Theorem~\ref{GreenConstructionAdjointCritical}, and $G(\cdot, \cdot)$ is Green's function for the operator $L$, then we have
\[
G(x,y)=g(y,x)
\]
for almost all $x,y\in\Omega$ with $x\neq y$.
\end{proposition}

\section{The gradient estimates}					\label{sec8}
This is an independent section dealing with the gradient bounds for the Green's function.
We begin with some definitions.
We shall denote
\[
\Omega_r(x):= \Omega \cap B_r(x).
\]
For a locally integrable function $g$ on $\Omega$, we shall say that $g$ is uniformly Dini continuous if the function $\varrho_g: \bR_+ \to \bR$ defined by
\[
\varrho_g(t):=\sup \Set{\,\abs{g(x)-g(y)}: x, y \in \Omega, \; \abs{x-y} \le t\,}.
\]
satisfies the Dini condition
\[
\int_0^1 \frac{\varrho_g(t)}{t} \,dt <+\infty.
\]
We shall say that $g$ is of \emph{Dini mean oscillation}
(in $\Omega$) if the function $\omega_g: \bR_+ \to \bR$ defined by
\[
\omega_g(r):=\sup_{x\in \overline{\Omega}} \fint_{\Omega_r(x)} \,\abs{g(y)-\bar {g}_{\Omega_r(x)}}\,dy \quad \left(\;\bar g_{\Omega_r(x)} :=\fint_{\Omega_r(x)} g\;\right)
\]
satisfies the Dini condition
\[
\int_0^1 \frac{\omega_g(t)}{t} \,dt <+\infty.
\]
It should be clear that if $g$ is uniformly Dini continuous, then it is of Dini mean oscillation and $\omega_g(r) \le \varrho_g(r)$.
It is worthwhile to note that if $\Omega$ is a Lipschitz domain and if $g$ is of Dini mean oscillation, then $g$ is uniformly continuous with a modulus of continuity determined by $\omega_g$.

For $k=1, 2, \ldots$, we say $\partial \Omega$ is $C^{k,\rm{Dini}}$ if for each point $x_0 \in \partial\Omega$, there exist $r>0$ independent of $x_0$ and a $C^{k, Dini}$ function (i.e., $C^k$ function whose $k$th derivatives are uniformly Dini continuous) $\gamma: \bR^{n-1} \to \bR$ such that (upon relabeling and reorienting the coordinates axes if necessary) in a new coordinate system $(x',x^n)=(x^1,\ldots,x^{n-1},x^n)$, $x_0$ becomes the origin and
\[
\Omega \cap B_r(0)=\set{ x \in B_r(0) : x^n > \gamma(x^1, \ldots, x^{n-1})},\quad  \gamma(0')=0.
\]

Now we state the main result of this section.
Recall that we consider the operator $L$ defined by
\[
Lu=-\dv(\mathbf{A}\nabla u+\vec bu)+\vec c \cdot \nabla u+du.
\]

\begin{theorem} \label{GradientBounds}
Let $\Omega$ be a bounded $C^{1,\rm{Dini}}$ domain and the coefficients of $L$ satisfy the hypotheses of Theorem \ref{GreenConstruction}.
Assume further that the coefficients $\mathbf{A}$ and $\vec b$ are of Dini mean oscillation in $\Omega$ and that $\vec c$, $d \in L^p$ with $p>n$.
Then the Green's function $G(x,y)$ satisfies the following gradient bound:
\begin{equation}					\label{eq0956f}
\abs{\nabla_x G(x,y)} \le C \abs{x-y}^{1-n},
\end{equation}
where $C$ depends on $n$, $\lambda$, $\omega_{\mathbf A}$, $\norm{\mathbf A}_{L^\infty}$, $\omega_{\vec b}$, $\norm{\vec b}_{L^\infty}$, $p$, $\norm{\vec c}_{L^p}$, $\norm{d}_{L^p}$, and $\partial \Omega$.
\end{theorem}
\begin{proof}
For fixed $x, y\in \Omega$ with $x \neq y$, set $r=\frac14 \abs{x-y}$ and note that $u=G(\cdot,y)$ is a weak solution of $Lu=0$ in $\Omega_{2r}(x)$ and $u=0$ on $\partial \Omega \cap B_{2r}(x)$.
By the proof of \cite[Theorem~1.8]{DEK17}, we see that $\nabla u$ is continuous in $\overline{\Omega_r}(x)$ and satisfies the estimate
\begin{equation}				\label{eq0940f}
\sup_{\Omega_r(x)} \,\abs{\nabla u} \le \frac{C}{r^{1+\frac{n}{2}}}\left( \int_{\Omega_{3r}(x)} \abs{u}^2  \right)^{\frac12}.
\end{equation}
Indeed, by  Lemma~2.41 in \cite{DEK17} (see also the paragraph before Proposition~2.14 in \cite{DEK17}) with proper scaling, we have
\[
\sup_{\Omega_r(x)}\, \abs{\nabla u} \le \frac{C}{r^{n}} \int_{\Omega_{2r}(x)} \abs{\nabla u}.
\]
Then by H\"older's inequality and Caccioppoli's inequality, we get \eqref{eq0940f}.
Now, \eqref{eq0956f} follows from \eqref{eq0940f} and the pointwise bound of $G$.
\end{proof}

We remark that, for the previous theorem, more restrictive assumptions on the coefficients of the equation are imposed, compared to the assumptions considered in the previous sections. This is due to the fact that we rely on the the results in \cite{DEK17}.


\end{document}